\documentclass[10pt,reqno]{amsart}
\usepackage[top=1in,bottom=1.2in,left=1.25in,right=1.25in]{geometry}
\usepackage{xr}
\usepackage{hyperref}
\usepackage{graphicx}
\usepackage[foot]{amsaddr}
\usepackage{amsmath}
\usepackage{amssymb}
\usepackage{amsthm}
\usepackage{xcolor}
\usepackage{float}
\usepackage{indentfirst}
\usepackage{placeins}
\usepackage{booktabs}
\usepackage{mathtools}
\usepackage{cleveref}
\usepackage{subcaption}
\usepackage{cleveref}
\usepackage{tikz}
\usepackage[backend=biber,
            style=numeric,
            sortcites=true]{biblatex}
\usepackage{algorithm} 
\usepackage{algpseudocode}

\newtheorem{theorem}{Theorem}[subsection]
\newtheorem{lemma}{Lemma}[subsection]
\newtheorem{definition}{Definition}[subsection]
\newtheorem{proposition}{Proposition}[subsection]
\newtheorem{corollary}{Corollary}[subsection]
\newtheorem{assumption}{Assumption}[subsection]

\theoremstyle{remark}
\newtheorem{remark}{Remark}

\newtheoremstyle{cited}%
  {3pt}
  {3pt}
  {\itshape}
  {}
  {\bfseries}
  {.}
  {.5em}
  {\thmname{#1} \thmnumber{#2} \thmnote{\normalfont#3}}

\theoremstyle{cited}

\renewcommand{\d}{\mathrm{d}}

\newcommand{\E}{\mathbb{E}}
\newcommand{\cE}{\mathcal{E}}

\newcommand{\Osc}{\mathrm{Osc}}
\newcommand{\osc}{\mathrm{osc}}

\newcommand{\bbm}{\begin{bmatrix}}
\newcommand{\ebm}{\end{bmatrix}}

\newcommand{\cX}{\mathcal{X}}
\newcommand{\cL}{\mathcal{L}}
\newcommand{\R}{\mathbb{R}}

\newcommand{\cO}{\mathcal{O}}

\newcommand{\cP}{\mathcal{P}}

\title[Fast transition operator learning]{Fast operator learning for mapping correlations }

\author[Y. Khoo]{Yuehaw Khoo}
\address{Department of Statistics and CCAM, University of Chicago, Chicago, IL, 60637}
\email{ykhoo@uchicago.edu}

\author[Y. Wang]{Yuguan Wang}
\address{Department of Statistics, University of Chicago, Chicago, IL, 60637}
\email{yuguanw@uchicago.edu}

\author[S. Yang]{Siyao Yang}
\address{Department of Statistics and CCAM, University of Chicago, Chicago, IL 60637}
\email{siyaoyang@uchicago.edu}



\begin{document}

\begin{abstract}
We propose a fast, optimization-free method for learning the transition operators of high-dimensional Markov processes. The central idea is to perform a Galerkin projection of the transition operator to a suitable set of low-order bases that capture the correlations between the dimensions. Such a discretized operator can be obtained from moments corresponding to our choice of basis without curse of dimensionality. Furthermore, by exploiting its low-rank structure and the spatial decay of correlations, we can obtain a compressed representation with computational complexity of order $\mathcal{O}(dN)$, where $d$ is the dimensionality and $N$ is the sample size. We further theoretically analyze the approximation error of the proposed compressed representation. We numerically demonstrate that the learned operator allows efficient prediction of future events and solving high-dimensional boundary value problems. This gives rise to a simple linear algebraic method for high-dimensional rare-events simulations. 

\end{abstract}

\maketitle

\section{Introduction}
Operator learning provides a data-driven framework for  partial differential equation solving, uncertainty quantification, and generative modeling~\cite{doi:10.1137/21M1401243,KOVACHKI2024419,BOULLE202483,zhang2025probabilisticoperatorlearninggenerative}. It aims to approximate an operator that maps input functions, such as initial condition, boundary data or coefficients, to the corresponding solutions, using data in the form of input–solution pairs. Once the operator is learned, one can evaluate solutions for new inputs without repeatedly solving the equations, making operator learning a popular framework for engineering applications, physical simulations, and scientific machine learning. In the context of rare-events simulation, a learned evolution operator contains the coarse grained dynamical information of a physical system. Many existing operator-learning methods work well for problems with low intrinsic dimensions, such as Markov state models~\cite{Schtte1999ADA,doi:10.1021/jp037421y,PANDE201099}, diffusion maps~\cite{coifman2006diffusion}, DeepONet~\cite{osti_2281727}, Fourier neural operator \cite{li2020fourier} and PCA-Net \cite{hesthaven2018non}. However, these approaches typically suffer from the {\it curse of dimensionality}, as their output representations in high dimensions often require exponentially large memory or sample complexity.


For high-dimensional problems, one common strategy is to represent the output solutions as Monte Carlo samples, and to learn the operators as neural networks such as normalizing flow~\cite{Kobyzev2020NormalizingFA,pmlr-v37-rezende15}, diffusion models \cite{ho2020denoising,song2020score,pmlr-v235-chen24n} and latent space simulator~\cite{D0SC03635H}  to overcome curse of dimensionality. However, training a neural network typically requires solving large-scale non-convex optimization problems, giving rise to a gap between actual error and the theoretical generalization error due to the inability to find the global minimizer. Another popular choice is the tensor network~\cite{Pirvu_2010,PhysRevA.74.022320,Lucke2022tgEDMD,PhysRevA.81.062337,khoromskij2011quantics}, which approximates a high-dimensional operator by contractions of small tensors. The output solutions can also be represented cheaply as tensor networks. A recent work~\cite{doi:10.1137/24M1654075} demonstrates that one can learn a tensor-network operator using linear algebraic operations instead of optimizations. However, when it comes to downstream tasks such as solving an elliptic PDE with the learned PDE operator (e.g. the committor function problem in rare events simulation~\cite{E2010TransitionpathTA}), one may still need to call some non-convex optimization procedure to obtain the PDE solution~\cite{CHEN2023111646}. 

\subsection{Our contribution}
In this work, we propose an optimization-free framework for Markov operator learning that requires only a small number of short-time trajectory simulations. The method begins by discretizing the transition operator as a $N_b\times N_b$ \emph{transition moment matrix}, which is a Galerkin projection of the transition operator $\mathcal{P}_t$ to a set of $d$-dimensional basis functions $\{\psi_i\}_{i=1}^{N_b}$:
    \begin{displaymath}
         M_{i,j} = \int \psi_i(x)\,(\mathcal{P}_t\psi_j)(x)\,\mu(x)\,\mathrm{d}x,
    \end{displaymath}
where $\mu$ is some given probability density. While similar operator-discretization ideas appear in computational chemistry such as dynamical Galerkin approximation (DGA) \cite{strahan2021long,thiede2019galerkin}, our main contribution is to generalize such a method to high-dimensional space without the curse of dimensionality. Our contributions are summarized as following:
\begin{enumerate}

    \item \textbf{Discretization capturing pairwise correlations between dimensions.} The naive choice of basis $\{\psi_i\}_{i=1}^{N_b}$  such as finite element type basis may have an exponential cost. In our method, we choose \emph{two-cluster basis functions}~\cite{doi:10.1098/rspa.2024.0001,chen2023,peng2023,khoo2025optimizationfreediffusionmodel,brydges1984short}, which can efficiently capture pairwise correlations between different dimensions. The size of the resulting transition moment matrix $M$ is $\mathcal{O}(d^2\times d^2)$ and each entry of $M$ can be estimated via Monte Carlo with sample complexity independent of the dimension $d$.
    \item \textbf{Linear-scaling algorithm in dimension.} To further reduce the complexity, we develop a fast linear-algebra-based algorithm to learn $M$ with $\mathcal{O}(d)$ memory and computational cost by estimating only linear number of entries in $M$, without using any optimization.
    \item \textbf{Approximation error guarantees.} The proposed compression algorithm is based on assuming a certain factorization of $M$, where $M$ is low-rank and its low-rank factors further admit some sparse-plus-low-rank factorization. We theoretically analyze the approximation error of $M$ with this factorization.
    \item \textbf{Solving high-dimensional problems without optimization.} Extensive numerical experiments demonstrate how the proposed method provides an optimization-less framework for  forecasting high-dimensional moments, probability densities,  and solving boundary value problems for rare-events simulations.
\end{enumerate}

\subsection{Organization}
The rest of the paper is organized as follows. In Section~\ref{sec:preliminaries}, we review the background on reversible Markov process and define the associated transition moment matrix. In Section~\ref{sec:entries}, we discuss the Monte Carlo method for evaluating its entries. In Section~\ref{sec:main_approach}, we present the proposed two-level compression scheme to fast form the transition moment matrix. In Section~\ref{sec:applications}, we provide the applications of the learned transition moment matrix in several representative high-dimensional problems. In Section~\ref{sec:theory}, we provide theoretical analysis for the proposed method under a finite-state Markov model. In Section~\ref{sec:numerical_results}, we report our numerical results. Finally in Section~\ref{sec:diccusion}, we conclude with remarks and directions for future work.

\section{Preliminaries}
\label{sec:preliminaries}
In this section, we review the background of the considered problem and method to be proposed.

\subsection{Reversible Markov process}
\label{sec:langevin}
We consider an ergodic  and reversible Markov process $\{X_t\}_{t\geqslant 0}$ on a $d$-site (i.e., $d$-dimensional) product space $\Omega = \cX^d$ where  $\cX$ denotes the univariate state space. The associated transition operator $\{\mathcal{P}_t\}_{t\geqslant 0}$ is defined by
\begin{align}\label{def:semigroup}
    (\mathcal{P}_t f)(x):=\E[f(X_t)\mid X_0=x],
\end{align}
for any observable $f:\Omega \rightarrow \R$, which characterizes  the dynamics of $X_t$.  The (infinitesimal) generator $\mathcal{L}$ is defined by
\begin{equation}\label{def:generator} 
    \mathcal{L} f := \lim_{t\to 0+} \frac{\mathcal{P}_t f - f }{t},
\end{equation}
and the transition operator can be expressed  through the exponential map
\begin{align*}
    \cP_t = e^{t\cL}, \quad \forall t\geqslant 0.
\end{align*}
We consider its projection onto a finite set of basis functions $\{\psi_i\}_{i=1}^{N_b} \subset L^2_\mu(\Omega)$ under some measure $\mu$, which yields the transition moment matrix $M_\mu^t \in \mathbb{R}^{N_b\times N_b}$ formulated by 
\begin{equation}\label{transition matrix}
    M_\mu^t(i,j) = \langle \psi_{i}, \cP_t  \psi_{j}\rangle_\mu   
\end{equation}
where the inner product w.r.t measure $\mu$ is defined by 
\begin{displaymath}
   \langle f, g\rangle_\mu  := \int_{\Omega}  f(x) g(x) \mu(x)\mathrm{d}x
\end{displaymath}
for $f,g\in L^2_\mu(\Omega)$. In particular, we denote $\langle f \rangle_\mu = \int_\Omega f(x) \mu(x)\mathrm{d}x$ for future use. The $(i,j)$-entry in~\eqref{transition matrix} represents the action of $\cP_t$ with lag time $t$ in the chosen subspace, and measures the time-lagged correlation between observables $\psi_i$ and $\psi_j$. The transition moment matrix thus provides a low-dimensional approximation of the transition operator.

In this work, we mainly consider two choices of the measure $\mu$: 
\begin{itemize}
    \item \textbf{Mean-field density:} $\mu$ takes a separable form
\begin{equation}
    \label{mean-field}
    \mu(x) = \mu_1(x_1)\, \mu_2(x_2) \cdots \mu_d(x_d),
\end{equation}
for one-dimensional marginal densities $\mu_i \in L^1(\mathcal{X})$. 
    \item \textbf{Equilibrium density:} Let $\rho_t(x)$ denotes the probability density of $X_t$ at time $t \geqslant 0$. 
For an ergodic and reversible Markov process, 
there exists a unique invariant equilibrium density $\rho_\infty$ such that 
\begin{equation}\label{eq_density}
    \rho_t \to \rho_\infty \quad \text{as}\quad  t \to \infty.
\end{equation}
\end{itemize}
The following standard properties of the generator $\mathcal{L}$ 
and the transition operator $\mathcal{P}_t$ under the equilibrium density $\rho_\infty$ 
will be used throughout this work.
\begin{proposition}
 \label{prop:PSD}
Let $\{X_t\}_{t \geqslant 0}$ be a reversible Markov process with invariant density $\rho_\infty$. 
Then, for all $f, g \in L^2_{\rho_\infty}(\Omega)$ and all $t \geqslant 0$, the following properties hold:
\begin{itemize}
    \item \textbf{Self-adjointness:}
    \begin{equation}\label{self-adjointness}
        \langle \mathcal{P}_t f, g \rangle_{\rho_\infty}
        = \langle f, \mathcal{P}_t g \rangle_{\rho_\infty},
        \qquad
        \langle \mathcal{L} f, g \rangle_{\rho_\infty}
        = \langle f, \mathcal{L} g \rangle_{\rho_\infty}.
    \end{equation}
    Hence, both $\mathcal{P}_t$ and $\mathcal{L}$ are self-adjoint operators on $L^2_{\rho_\infty}(\Omega)$.
    
    \item \textbf{Positive semi-definiteness:} The generator $\mathcal{L}$ is negative semi-definite:
    \begin{equation}\label{L negative}
        \langle f, \mathcal{L} f \rangle_{\rho_\infty} \leqslant 0.
    \end{equation}
   Hence, $-\mathcal{L}$ and transition moment matrix $M^t_{\rho_\infty}$ are positive semi-definite (PSD). 
    \end{itemize}
\end{proposition}
In the following subsection, we discuss the choice of basis functions to form the transition moment matrix \eqref{transition matrix}.

\subsection{Cluster basis}
\label{sec:moment}
Let $\{\phi_j^i\}_{j=1}^n$ be a univariate orthonormal basis on $\mathcal{X}$ acting on the $i$-th dimension, 
with $\phi_1^i$ being the constant function (e.g. Fourier, Legendre basis). 
The set of \emph{$k$-th order cluster basis functions} (or simply, \emph{$k$-cluster bases}) is defined as
\begin{align}
\label{eqn:cluster_basis_set}
   \Big\{
        \phi_{j_1}^{i_1} \cdots \phi_{j_k}^{i_k}
        :
         (j_1, \ldots, j_k) \in [n]^k,\;
         (i_1, \ldots, i_k) \in [d]^k
    \Big\},
\end{align}
where $[n] = \{1, 2, \ldots, n\}$. 

Throughout this paper, we use two-cluster bases (i.e., $k=2$) for our proposed algorithm, 
and treat the transition moment matrix $M_\mu^t$ as a $(dn)^2 \times (dn)^2$ matrix. For the sake of clarity, we introduce the following linear index for our basis such that 
\begin{equation}\label{two cluster linear}
    \psi_j(x)= \phi_{b_1(j)}^{s_1(j)}(x_{s_1(j)}) \cdot \phi_{b_2(j)}^{s_2(j)}(x_{s_2(j)}),\qquad j=1,\ldots,(dn)^2
\end{equation}
where $b(j) = (b_1(j),b_2(j)) \in [n]\times [n]$ are the indices for the order of the basis, and $s(j) = (s_1(j),s_2(j))\in [d]\times [d]$ are the indices for the sites. The set of two-cluster basis is denoted by $\{\psi_j\}_{j=1}^{(dn)^2}$. An example of two-cluster basis function is given in \Cref{app:idx}.

Compared with the standard tensor-product basis, which leads to an exponentially large matrix dimension, 
the use of cluster bases effectively mitigates the curse of dimensionality. 
Moreover, cluster bases have been demonstrated to be highly effective in representing local correlations in many-body physical systems. Related ideas also appear in other basis-expansion frameworks that group variables and exploit the interaction structure in many-body systems, such as the atomic cluster expansion (ACE) \cite{drautz2019atomic,dusson2022atomic} and the Mayer cluster expansion \cite{brydges1984short,meneghelli2014mayer}. 


\begin{remark}
    In the definition \eqref{two cluster linear} (see also the example in \Cref{app:idx}), 
different indices $j$ may correspond to the same cluster basis function. More specifically, there exist $j_1\neq j_2 \in [(dn^2)]$ such that 
\begin{displaymath}
    s_1(j_1) = s_2(j_2), \quad b_1(j_1) = b_2(j_2).
\end{displaymath}
Consequently, the resulting two-cluster functions $\psi_{j_1}(x)$ and $\psi_{j_2}(x)$ are identical. To obtain a strictly unique representation, 
one could impose an ordering constraint on the site indices, 
for example, one can impose $i_1 < i_2 < \cdots < i_k$ in \eqref{eqn:cluster_basis_set}. However, we deliberately avoid such a restriction, 
as it would destroy the natural tensor or matrix structure of the basis functions, which is crucial for our proposed algorithm. 
\end{remark}

\section{Evaluation of entries of transition moment matrix}
\label{sec:entries}
In this section, we discuss the Monte Carlo method for evaluating the entries of the transition moment matrix \eqref{transition matrix}
\begin{displaymath}
    M_\mu^t(j_1,j_2) 
    = \int_\Omega 
    \psi_{j_1}(x)\,
    \mathbb{E}\!\left[ \psi_{j_2}(X_t) \mid X_0 = x \right]
    \, \mu(x)\mathrm{d}x,
\end{displaymath}
which involves high-dimensional expectation and integration. A practical approach is to use a Monte Carlo estimator. Specifically, we draw $N_{\text{src}}$ i.i.d.\ source samples $\{x^{(i)}\}_{i=1}^{N_{\text{src}}} \sim \mu$ to approximate the outer expectation, 
and for each $x^{(i)}$, we simulate $N_{\text{traj}}$ independent trajectories 
$\{y^{(i,l)}\}_{l=1}^{N_{\text{traj}}}$ starting from $x^{(i)}$ 
to approximate the conditional expectation. 
This yields the empirical approximation
\begin{align}
\label{eqn:moment_Mc_approx}
   M_\mu^t(j_1,j_2) 
    \approx  
    \frac{1}{N_{\text{src}} N_{\text{traj}}} 
    \sum_{i=1}^{N_{\text{src}}} \psi_{j_1}(x^{(i)}) 
    \sum_{l=1}^{N_{\text{traj}}}  \psi_{j_2}(y^{(i,l)}).
\end{align}

In this work, we also consider a formulation of the transition moment matrix 
for boundary value problems. 
Let $D \subset \Omega$ be a domain, 
and assume the two-cluster basis functions satisfy $\psi_{j}(x) = 0$ for $x \in D$. 
Define the stopping time
\begin{equation}\label{stop time}
        \tau_D = \inf\{\, t \geqslant  0 : X_t \in D \,\},
\end{equation}
representing the first time the process $X_t$ enters $D$. 
The corresponding stopped transition operator  is defined as
\begin{align}
\label{eqn:semigroup_stop}
    \mathcal{P}_{t \wedge \tau_D} f(x) 
    := \mathbb{E}\!\left[ f(X_{t \wedge \tau_D}) \mid X_0 = x \right].    
\end{align}
The transition moment matrix associated with this operator is
\begin{equation}
\label{transition matrix true stop}
    M_\mu^{t \wedge \tau_D}(j_1,j_2) 
    = \int_{\Omega \setminus D} 
    \psi_{j_1}(x)\,
    \mathbb{E}\!\left[ \psi_{j_2}(X_{t \wedge \tau_D}) \mid X_0 = x \right]
    \mu(x)\mathrm{d}x.
\end{equation}
Its Monte Carlo approximation then takes the form
\begin{equation}
\label{eqn:moment_Mc_approx_stop}
    M_\mu^{t \wedge \tau_D}(j_1,j_2) 
    \approx  
    \frac{1}{N_{\text{src}} N_{\text{traj}}} 
    \sum_{i : \,x^{(i)} \notin D} \psi_{j_1}(x^{(i)}) 
    \sum_{l :\, y^{(i,l)} \notin D} \psi_{j_2}(y^{(i,l)}).
\end{equation}
Since there are $(dn)^4$  entries in the transition moment matrix and let $\zeta$ be the computational cost to generate a single trajectory, the total computational complexity of this naive approach scales as
$$
\cO(d^4n^4\zeta N)
$$
where $N=N_{\text{src}}N_{\text{traj}}$ is the total sample size. Therefore, both memory and computational complexity scale as $\mathcal{O}(d^4)$, which can be prohibitively large in high-dimensional cases. To alleviate this, we propose an efficient compression  in the following section.

\section{Fast construction of transition moment matrix}
\label{sec:main_approach}
In this section, we present a fast method for compressing the transition moment matrix~\eqref{transition matrix} formed under the two-cluster bases \eqref{two cluster linear}. The method exploits the intrinsic structure of the transition moment matrix and performs efficient compression using randomized numerical linear algebra techniques. In Section~\ref{sec:overview}, we provide a high-level overview of the proposed approach. In Section~\ref{sec:first_compression}, we discuss the global compression of the transition moment matrix into a low-rank format. In Section~\ref{sec:second_compression}, we present the fast construction and compression of the resulting low-rank factors. Finally, in Section~\ref{sec:complexity}, we provide additional implementation details for both compression stages.

\subsection{Overview}
\label{sec:overview}
Our proposed compression algorithm consists two steps. 
\begin{itemize}
    \item[(1)] Low-rank factorization of transition moment matrix (Section~\ref{sec:first_compression}): We exploit the global low-rankness of $M_\mu^t$ and form a  factorization of $M_\mu^t$ based on {\it CUR approximation} (or {\it cross approximation})~\cite{doi:10.1073/pnas.0803205106}, which reduces computation to a small set of selected rows and columns, and the submatrix formed by their intersection.  
    \item[(2)] Fast construction of low-rank factors (Section~\ref{sec:second_compression}): We use the underlying decay of correlations between coordinates to further decompose each row or column in CUR approximation as the sum of a sparse factor and a low-rank factor. All entries of these factors are then evaluated using the Monte Carlo method introduced in \Cref{sec:entries}.
\end{itemize}
The two compression steps are schematically illustrated in Figure~\ref{fig:compression}.
With the proposed two-level compression strategy, the transition moment matrix is eventually stored as sparse and low-rank factors with $\mathcal{O}(d)$ memory cost, where $d$ denotes the dimensionality of the system. Furthermore, we can achieve an optimal computational cost of $\mathcal{O}(dN)$, where $N$ is the number of Monte Carlo samples, to obtain these factors. A summary of the procedures is provided in Algorithm~\ref{alg:two_level}.

\begin{figure}[ht]
\centering
\begin{tikzpicture}[scale=0.99]
\node[inner sep=0] at (0,0)
{\includegraphics[width=0.6\textwidth, trim = 0 50 0 0, clip]{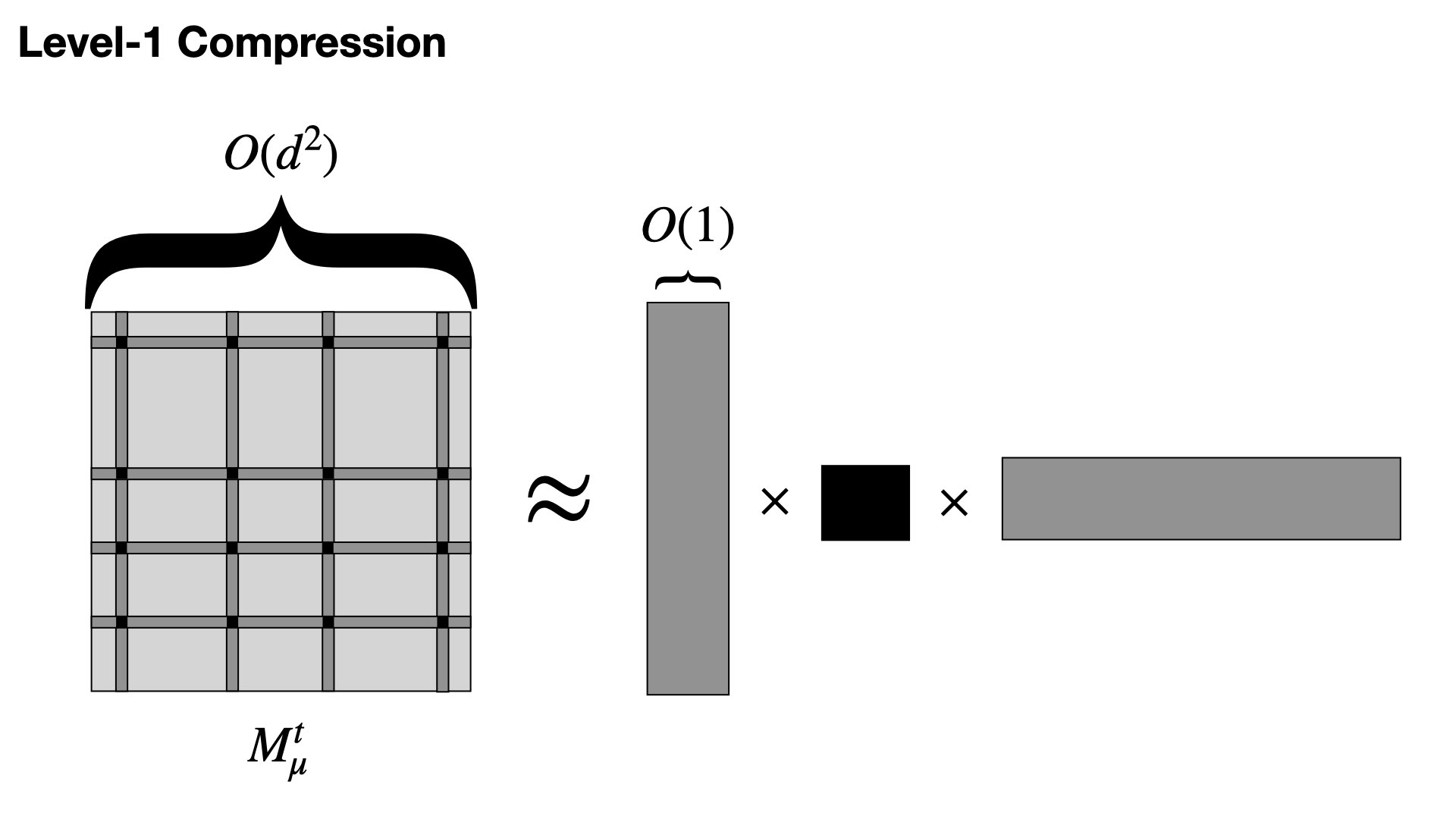}};
\node[inner sep=0] at (0,-5)
{\includegraphics[width=0.6\textwidth, trim = 0 300 0 0, clip]{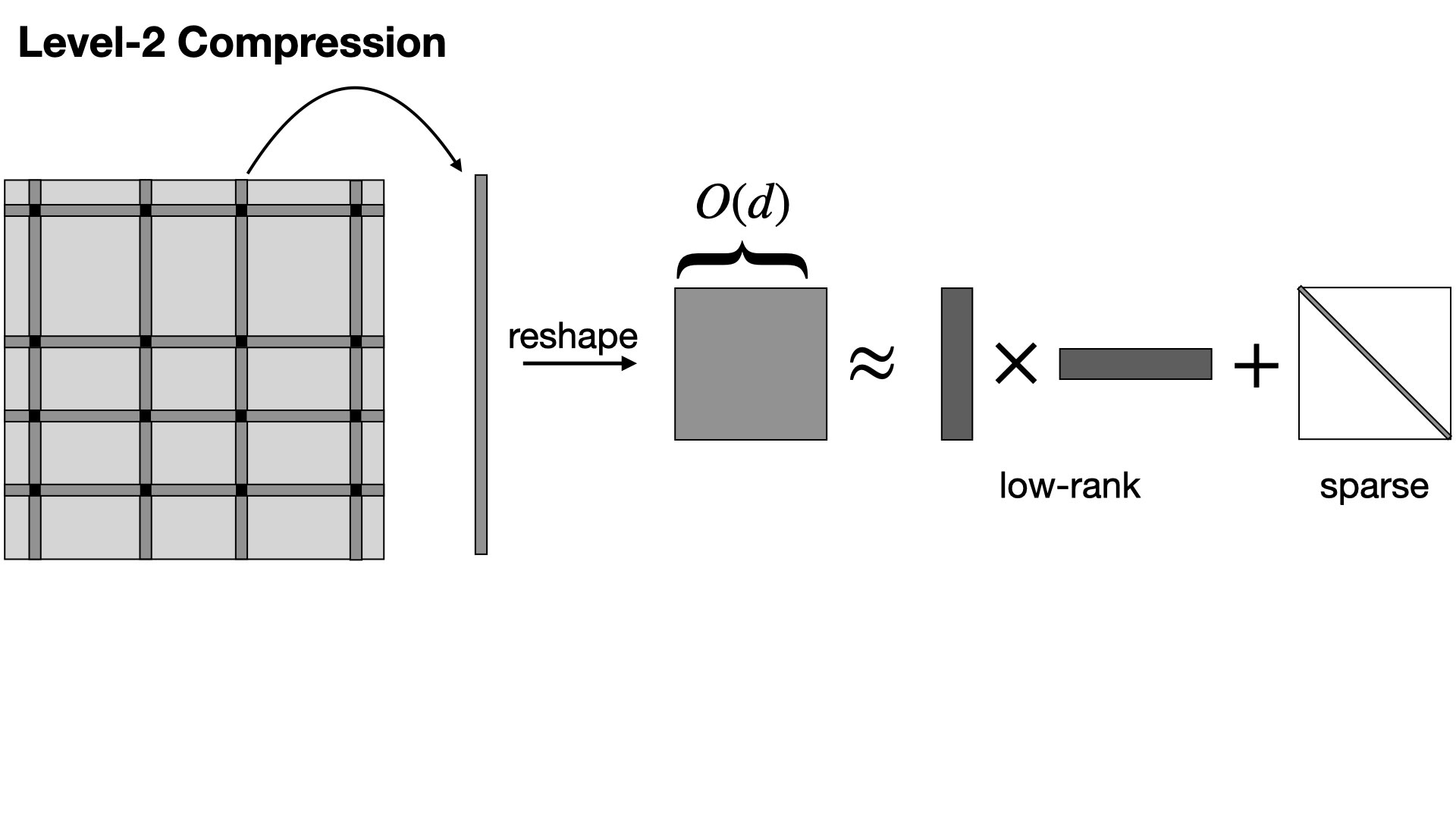}};
\end{tikzpicture}
\caption{Two-level compression of transition moment matrix.}
\label{fig:compression}
\end{figure}

\begin{algorithm}[ht!]
\caption{Two-level compression of transition moment matrix }
\label{alg:two_level}
\begin{algorithmic}[1]
\Require Transition moment matrix $M^t_\mu\in \R^{(dn)^2\times (dn)^2}$ defined in \eqref{transition matrix}; target rank $r_1$, $r_2$ for low-rank factors; target bandwidth $\delta$ for sparse factors.
\Ensure Compressed factorization of $M^t_\mu$.
\State Apply CUR approximation to transition moment matrix: 
\begin{equation}\label{eq:cur}
    M^t_\mu \approx M^t_\mu(:,J)  \left( M^t_\mu(I,J) \right)^\dag  M^t_\mu(I,:),
\end{equation}
where $I,J\subset [(dn)^2]$ are respectively selected $r_1$ column and row indices, $\left( M^t_\mu(I,J) \right)^\dag$ is the Moore–Penrose pseudoinverse of the submatrix $M^t_\mu(I,J)$ formed by the intersection of selected columns $M^t_\mu(:,J)$ and rows $M^t_\mu(I,:)$. 
\State Approximate each selected column  $M^t_\mu(:,j) \in \R^{(dn)^2\times 1}$ or row $M^t_\mu(j,:) \in \R^{1\times (dn)^2}$ by the sum $P_j + \text{vec}(Q_j)$, where $P_j$ is a sparse vector with $\mathcal{O}(dn\delta)$ nonzero entries, and $\text{vec}(Q_j)$ is the vectorization of a low-rank matrix $Q_j$ represented by a rank-$r_2$ factorization.
\State Evaluate each entry in the obtained factors by Monte Carlo method in \Cref{sec:entries}. 
\end{algorithmic}
\end{algorithm}

\subsection{Low-rank factorization of transition moment matrix}
\label{sec:first_compression}
In this step, we aim to obtain a low-rank factorization of the transition moment matrix $M^t_\mu$ using the CUR approximation. Specifically, we select $r_1$ pivoted rows and $r_1$ pivoted columns from $M^t_\mu$ and   approximate it as \eqref{eq:cur}. A straightforward approach is to perform uniform sampling without replacement, which leads to the Nystr\"om method \cite{williams2000using}. However, purely random pivot sampling  may require a relatively large target rank $r_1$ to achieve satisfactory accuracy. To improve efficiency and accuracy, we use the \emph{maximum-volume} (MaxVol) based adaptive cross approximation \cite{Goreinov2001TheMC,submatrix,AllenLaiShen2024Maxvol}, 
which iteratively refines the selected row and column indices 
to identify a near-optimal set of size $r_1$. 
Further details of MaxVol-based adaptive cross approximation algorithm are provided in~\Cref{sec:complexity}. 

The effectiveness of the CUR approximation in this setting relies on the fact that $M_\mu^t$ becomes low-rank for sufficiently large lag time $t$.  The theoretical justification of this property is deferred to~\Cref{sec:main results}.
Compared with other low-rank decompositions such as the singular value decomposition (SVD), a key advantage of the CUR approximation is that all three factors are submatrices of the original matrix $M_\mu^t$. 
This features enable us to  exploit the intrinsic local structure of $M_\mu^t$ 
and to construct each factor efficiently via the second-level compression procedure to be  described in the next subsection.

\subsection{Fast construction of low-rank factors}
\label{sec:second_compression}
We now present the second-level compression applied to the CUR factors. The key observation underlying this step is that the correlation between cluster basis functions decays rapidly as their spatial separation increases. 

To illustrate the idea, consider first a special case in which $\mu$ is the mean-field density defined in~\eqref{mean-field}, and the underlying many-body system is fully uncorrelated, that is, the dynamics evolve independently across sites. In this setting, the entry of the transition moment matrix~\eqref{transition matrix} constructed using the two-cluster bases~\eqref{eqn:cluster_basis_set} satisfies the factorization
\begin{equation}\label{zero_correlation}
    \langle \psi_{i} ,\cP_t \psi_{j} \rangle_{\mu} =   \langle \psi_{i} \rangle_{\mu}  \langle \cP_t \psi_{j} \rangle_{\mu}    
\end{equation}
if  the site indices $s_1(i),s_2(i),s_1(j),s_2(j)$ are all distinct. 

Motivated by this observation, we now turn to systems with short-range interactions. Specifically, suppose the generator introduced in~\eqref{def:generator} decomposes into local components:
\begin{equation}\label{L decompose}
    \cL = \sum_{i =1}^d \cL_i
\end{equation}
where the local operator $\cL_i$ modifies only the coordinates in the neighborhood of site $i$, and its action depends only on $x_{\mathcal{N}(i,\kappa)}$ where 
\begin{displaymath}
    \mathcal{N}(i,\kappa)  = \{i' \in [d]  :  |i'-i| \leqslant \kappa \}
\end{displaymath}
where $\kappa>0$ is an integer representing the range of interaction.  
This assumption is satisfied by a broad class of lattice systems in physics 
that exhibit short-range interactions~\cite{10.1214/aop/1176993067,Caputo2004SpectralGap,10.1214/aop/1041903209,cmp/1104253633}. In the limiting case $\kappa=0$, the system reduces to the uncorrelated example discussed above.

For systems with finite interaction range $\kappa>0$, while the factorization \eqref{zero_correlation} on longer holds exactly, one may still expect it holds approximately as 
\begin{equation}\label{decay_of_correlation_approx}
    \langle \psi_{i} ,\cP_t \psi_{j} \rangle_{\mu} \approx   \langle \psi_{i} \rangle_{\mu}  \langle \cP_t \psi_{j} \rangle_{\mu}
\end{equation}
for given two-cluster bases $\psi_{i}$ and $\psi_{j}$, if the distance between the two sites defined by 
\begin{equation}\label{site distance}
     \xi(i,j) := \min\{|s_1(i)-s_1(j)|,|s_2(i)-s_2(j)|,|s_1(i)-s_2(j)|,|s_2(i)-s_1(j)|\}
\end{equation}
is sufficiently large. The intuition is that for moderately small lag time $t$, the influence cone of $\mathcal{P}_t \psi_{j}$ remains localized and does not propagate to the sites indexed by $s(i)$. Consequently, the correlation $\langle \psi_{i}, \mathcal{P}_t \psi_{j} \rangle_{\mu} 
    - 
    \langle \psi_{i} \rangle_{\mu} 
      \langle \mathcal{P}_t \psi_{j} \rangle_{\mu} $ decays as the distance $\xi(i,j)$ grows. This phenomenon is known as the \emph{decay of correlation}~\cite{liggett1997interacting}. 
In many models, such decay is exponential in $\xi(i,j)$ and grows exponentially in lag time $t$, in accordance with {\it Lieb--Robinson-type} bounds~\cite{capel2025decay,chen2023speed,PhysRevLett.104.190401}. In \Cref{sec:theory}, we rigorously derive an upper bound on the correlation in a simplified finite-state Markov system with equilibrium density $\mu = \rho_\infty$, which is another case we mainly consider in this work.  
Our bound exhibits exponential decay in the spatial distance $\xi(i,j)$ and, importantly, remains uniform in $t$, i.e., it does not deteriorate as the lag time increases.

The approximation \eqref{decay_of_correlation_approx} leads to a natural decomposition of each row or column sampled in CUR approximation into a sparse factor and a low-rank factor. As an example, consider a sampled column 
$M_\mu^t(:, j) \in \mathbb{R}^{(dn)^2 \times 1}$, for some sufficiently large bandwidth $\delta > 0 $, its $i$-th entry admits the approximation 
\begin{equation}
\label{eqn:G_approx}
\begin{split}
    M_\mu^t(i, j)= \ &  \langle \psi_i, \cP_t \psi_j \rangle_\mu\\
    \approx \ &  \langle \psi_i \rangle_\mu \cdot \langle \cP_t \psi_j \rangle_\mu  \qquad \qquad \qquad \  \text{for~} \xi(i,j)> \delta\\
     = \ &  \left\langle \phi^{s_1(i)}_{b_1(i)},\phi^{s_2(i)}_{b_2(i)} \right\rangle_\mu \cdot \langle \cP_t \psi_j \rangle_\mu \\ 
    \approx \ & \langle \phi^{s_1(i)}_{b_1(i)} \rangle_\mu \cdot \langle \phi^{s_2(i)}_{b_2(i)} \rangle_\mu \cdot \langle \cP_t \psi_j \rangle_\mu \quad  \text{for~} |s_1(i)-s_2(i)|> \delta.
\end{split}
\end{equation}
The second approximation above corresponds to the case for one-cluster bases with $t = 0$ in \eqref{decay_of_correlation_approx}. This suggests that the vector $M_\mu^t(:, j)$ can be decomposed into 
\begin{equation}
\label{col_sparse_lr_decomp}
M_\mu^t(:, j) = P_j+ \text{vec}( Q_j )
\end{equation}
where $P_j \in \R^{(dn)^2\times 1}$ satisfying
\begin{align}
\label{eqn:G_sparse}
P_j(i)=\begin{cases}
    M_\mu^t(i,j)-  \langle \phi^{s_1(i)}_{b_1(i)} \rangle_\mu \cdot \langle \phi^{s_2(i)}_{b_2(i)} \rangle_\mu \cdot \langle \cP_t \psi_j \rangle_\mu, & \min(\xi(i,j),|s_1(i)-s_2(i)|)\leqslant \delta\\
    0, & \text{otherwise}
\end{cases}
\end{align}
is a sparse vector with $\cO(dn\delta)$ nonzero entries. The residual $\text{vec}( Q_j )$ is the vectorization of matrix $Q_j \in \R^{dn\times dn}$ which, by construction of $P_j$ and  \eqref{eqn:G_approx}, satisfies 
\begin{equation}
 \label{eqn:rank-one-mat}
Q_j((k_1,l_1),(k_2,l_2))\approx \langle \phi^{k_1}_{l_1} \rangle_\mu \cdot \langle \cP_t \psi_j \rangle_\mu \cdot \langle \phi^{k_2}_{l_2} \rangle_\mu, 
\end{equation}
where  $(k_1,l_1),(k_2,l_2)\in [d]\times [n]$ denote composite indices. This suggests that $Q_j$ can be well approximated by a rank-one matrix, which is demonstrated numerically in \Cref{sec:sanity_check}. Based on these observations, we therefore compress $Q_j$, by first reshaping $M^t_\mu(:,j) - P_j$ into a $dn\times dn$ matrix, and then perform a low-rank approximation to the resulting matrix. The detailed procedure is summarized in \Cref{alg:slice_decompose}. The compression of a row of $M_\mu^t$ follows the same idea and the corresponding algorithm is given by \Cref{alg:slice_decompose_row}.

\begin{algorithm}[ht!]
\caption{Factorization of a column}
\label{alg:slice_decompose}
\begin{algorithmic}[1]
\Require Column index $j \in [(dn)^2]$; bandwidth $\delta$ for sparse factor; target rank $r_2$ for low-rank factors.
\Ensure Sparse vector $P_j \in \R^{(dn)^2}$ and low-rank factors $U \in \R^{dn \times  r_2}, V \in \R^{ r_2 \times dn}$, such that $M_\mu^t(:,j) \approx P_j+\mathrm{vec}(UV)$. 
\State Form the sparse factor $P_j$ by~\eqref{eqn:G_sparse}. 
\State Reshape $M_\mu^t(:,j)-P_j \in \R^{(dn)^2}$ into a matrix $A\in \R^{dn\times dn}$. Perform rank-$r_2$ approximation $A\approx UV$ (e.g., CUR, SVD).  
\end{algorithmic}
\end{algorithm}

\begin{algorithm}[ht!]
\caption{Factorization of a row}
\label{alg:slice_decompose_row}
\begin{algorithmic}[1]
\Require Row index $i \in [(dn)^2]$; bandwidth $\delta$ for sparse factor; target rank $r_2$ for low-rank factors.
\Ensure Sparse vector $P_i \in \R^{(dn)^2}$ and low-rank factors $U \in \R^{dn \times  r_2}, V \in \R^{ r_2 \times dn}$, such that $M_\mu^t(i,:) \approx P_i+\mathrm{vec}(UV)$. 
\State Form the sparse factor $P_i$ by
\begin{align*}
P_i(j)=\begin{cases}
    M_\mu^t(i,j)-  \langle \cP_t \phi^{s_1(j)}_{b_1(j)} \rangle_\mu \cdot \langle \cP_t \phi^{s_2(j)}_{b_2(j)} \rangle_\mu \cdot \langle \psi_i \rangle_\mu, & \min(\xi(i,j),|s_1(j)-s_2(j)|)\leqslant \delta\\
    0, & \text{otherwise}
\end{cases}
\end{align*}
\State Reshape $M_\mu^t(i,:)-P_i \in \R^{(dn)^2}$ into a matrix $A\in \R^{dn\times dn}$. Perform rank-$r_2$ approximation $A\approx UV$ (e.g., CUR, SVD).  
\end{algorithmic}
\end{algorithm}

\subsection{Complexity analysis and implementation details}
\label{sec:complexity}
\subsubsection{Complexity analysis}
We now analyze the overall memory and computational cost for forming the transition moment matrix 
$M_\mu^t$ defined in~\eqref{transition matrix} 
using the proposed two-level compression scheme introduced in Section~\ref{sec:main_approach}. 
\begin{itemize}
    \item \emph{Memory cost}: In Step~1 of Algorithm~\ref{alg:two_level}, 
we apply a rank-$r_1$ CUR approximation 
to decompose $M_\mu^t \approx C U R$, 
where only the cross matrix $U \in \mathbb{R}^{r_1 \times r_1}$ is explicitly constructed, 
requiring  $\mathcal{O}(r_1^2)$ entries. 
In Step~2, after reshaping into a matrix, each column of $C \in \mathbb{R}^{dn \times r_1}$ 
(or equivalently, each row of $R \in \mathbb{R}^{r_1 \times dn}$) is further decomposed into a sparse component 
with $\mathcal{O}(dn \delta)$ nonzero entries to be evaluated,  
and a low-rank component which is again approximated by a rank-$r_2$ CUR factorization where $\mathcal{O}(dnr_2)$ entries need to be evaluated. Hence, the total number of matrix entries that must be computed and stored is
\begin{displaymath}
        \mathcal{O}\!\left(r_1^2 + d n r_1 (r_2 + \delta)\right),
\end{displaymath}
which determines the overall memory scaling of the compressed transition moment matrix. If we treat the parameters $n$, $r_1$, $r_2$ and $\delta$ as constants, storing $M_\mu^t$ needs memory cost $\mathcal{O}(d)$, which is linear with the dimensionality. In contrast, explicitly forming the entire transition moment matrix would incur a memory cost of $\mathcal{O}(d^4)$, as explained in \Cref{sec:entries}, which is impractical for high-dimensional problems.
\item \emph{Computational cost}: Each entry is evaluated via Monte Carlo integration 
using the estimators in~\eqref{eqn:moment_Mc_approx} 
or~\eqref{eqn:moment_Mc_approx_stop}. 
For computational efficiency, 
we employ a common set of samples 
$\{x^{(i)}, \{y^{(i,j)}\}_{j=1}^{N_{\text{traj}}}\}_{i=1}^{N_{\text{src}}}$ 
for all entries of $M_\mu^t$. Recall that $\zeta$ denotes the cost of simulating one trajectory 
of the Markov process up to lag time $t$ and $N = N_{\text{src}} N_{\text{traj}}$ is the total number of trajectories.  The total cost of generating all samples is then 
$\mathcal{O}(\zeta N)$.
Once the trajectories are available, 
evaluating each entry via Monte Carlo integration 
requires $\mathcal{O}(N)$ operations according to \eqref{eqn:moment_Mc_approx} 
or \eqref{eqn:moment_Mc_approx_stop}. 
The overall computational cost of the proposed algorithm is therefore
\begin{displaymath}
        \mathcal{O}\!\left(\zeta N + r_1^3+[r_1^2 + d n r_1 (r_2 + \delta)] N\right).
\end{displaymath}
For systems with finite-range interactions, which are the primary focus of this work, the cost of simulating one trajectory 
typically scales linearly with the dimension, i.e., $\zeta = \mathcal{O}(d )$. If again the parameters $n$, $r_1$, $r_2$, and $\delta$ are treated as fixed constants, the overall computational cost for constructing $M_\mu^t$ is 
\begin{equation}\label{cost}
     \mathcal{O}(d N),
\end{equation}
which achieves linear scaling in both the system dimension $d$ and the sample size $N$. Again, directly constructing the full transition matrix  with two-cluster basis
would require $\mathcal{O}(d^4 N)$ operations  as discussed in \Cref{sec:entries}, 
rendering such an approach infeasible for high-dimensional systems.
\end{itemize}

\subsubsection{Implementation details}
\label{sec:implementation_details}
We conclude this section with several remarks on the practical aspects of implementation.
\begin{itemize}
    \item \emph{MaxVol-based adaptive cross approximation for improved pivot selection}: As discussed earlier, we employ MaxVol-based adaptive cross approximation algorithm to obtain better pivoted rows and columns. Starting from a small set of randomly sampled indices, this algorithm iteratively adds more rows and columns indices as pivots such that the volume (absolute value of determinant) of the intersection submatrix becomes comparatively large among all possible submatrices, which typically leads to a more accurate CUR approximation \cite{Kressner,AllenLaiShen2024Maxvol,submatrix,submatrix}. In practice, we set the number of iterations $n_{\mathrm{iter}}$ to be a $\mathcal{O}(1)$-constant, ensuring that the overall computational cost remains the same as in~\eqref{cost}. 
   \item \emph{CUR approximation for PSD matrices}: 
When $\mu$ is chosen as the equilibrium density (see the applications in \Cref{sec:committor}), 
the resulting transition moment matrix \eqref{transition matrix} is PSD by \Cref{prop:PSD}. 
In this setting, there exists a maximum-volume submatrix that is principal \cite{CORTINOVIS2020251}. 
Therefore, in Step~1 of \Cref{alg:two_level}, the same set of indices should be used for both the row and column selections.  
\item \emph{Higher-order cluster approximation via combining neighboring dimensions}: 
The proposed algorithm is based on the two-cluster basis ansatz \eqref{two cluster linear}. 
However, many-body systems with strong interactions may require higher-order cluster bases to achieve satisfactory accuracy. 
Within the current framework, this can be accomplished by grouping neighboring dimensions. 
For example, if the total dimension $d$ is even, one may combine every two adjacent dimensions and apply the change of variables 
$(y_1, y_2, \ldots, y_{d/2}) = (x_{1:2}, x_{3:4}, \ldots, x_{d-1:d})$. 
The proposed algorithm can subsequently be carried out with respect to these grouped coordinates. In the limiting scenario where all coordinates are combined into a single multi-index, the resulting cluster bases coincide with the conventional tensor-product basis. 
\item \emph{Efficient matrix-vector multiplication}: A common computational task, as will be discussed in the next section, is to apply the compressed transition moment matrix $\tilde M_\mu^t \in \mathbb{R}^{(dn)^2\times (dn)^2}$ to a vector in $\mathbb{R}^{(dn)^2\times 1}$. Without compression, this operation would require $\mathcal{O}(d^4)$ cost. In contrast, our compressed representation enables a much more efficient computation: each column or row of the low-rank factors is reconstructed from its sparse-plus-low-rank form at a total cost of $\mathcal{O}(d)$, after which the resulting low-rank factors of $M_\mu^t$ are applied sequentially to the input vector with a cost of $\mathcal{O}(d^2)$.
\end{itemize}

\section{Applications}
\label{sec:applications}
In this section, we outline frameworks for several high-dimensional tasks using the transition moment matrix, including moment prediction in ~\Cref{sec:backward}, density prediction in Section~\ref{sec:forward}, and boundary value problem  in Section~\ref{sec:committor}.

\subsection{Moment prediction}
\label{sec:backward}
We begin by applying the compressed transition moment matrix to the task of {\it moment prediction}. Given a function $g(x)$ defined on $\Omega$, our goal is to predict its moment/expectation
\begin{equation}\label{moment prediction}
    u_{t} = \cP_t g 
\end{equation}
at a future time $t$. The moment prediction problem is related to solving the Kolmogorov backward equation with terminal time condition
\begin{align}
\label{eqn:kbe}
\begin{cases}
    \frac{\partial u_t}{\partial t} +\cL u=0,& t\in [0,T]\\ 
    u_T=g
\end{cases}
\end{align}
whose solution is given by $u_t=\cP_{T-t}g$.

To perform the moment prediction in our framework, we approximate the solution $u_t$ by the two-cluster bases as 
\begin{equation*}
    u_t (x)= \mu(x)  \cdot \sum_{j=1}^{(dn)^2} c^t_{j} \psi_{j}(x)
\end{equation*}
for some measure $\mu$, where $c_j^t$ are the corresponding time-dependent coefficients. To determine these coefficients, we integrate against $\psi_{j'}$ to obtain the weak form
\begin{displaymath}
   \langle u_{t }, \psi_{j'} \rangle =   \langle \cP_t g,  \psi_{j'}  \rangle.
\end{displaymath}
After projecting the initial data $g(x)$ onto the bases to obtain $c^0_{j}$, this relation becomes
\begin{align*}
 \sum_{j=1}^{(dn)^2}     \langle  \psi_{j} ,  \psi_{j'} \rangle_\mu  c^{t}_{j}   = \sum_{j=1}^{(dn)^2}  \langle  \cP_t \psi_{j} , \psi_{j'}   \rangle_\mu c^0_{j}.
\end{align*}
In matrix form, this reads 
\begin{align*}
     M_\mu^{0} c^{t}   =  M_{\mu}^{t} c^{0},
\end{align*}
where the vector $c^\tau$ collects the coefficients at time $\tau = 0,t$. In practice, we approximate $ M_{\mu}^{t}$ by its compressed representation $\tilde M_{\mu}^{t}$, obtained according to our proposed algorithm in \Cref{sec:main_approach}. Leveraging the low-rank and sparse structure of $\tilde M_{\mu}^{t}$, the moment prediction can then be computed by
\begin{equation*}
    c^{t} = \left(  M_\mu^{0}\right) ^{-1} \cdot \left( \tilde M_{\mu}^{t} c^{0} \right),
\end{equation*}
where the matrix–vector operations between $\tilde M_{\mu}^{t}$ and $c^0$ are performed using the sparse and low-rank factors in the compressed representation, avoiding the need to form or invert the full dense matrix explicitly as mentioned at the end of \Cref{sec:implementation_details}. 

\subsection{Density prediction by adjoint operator}
\label{sec:forward}
A natural extension of \Cref{sec:backward} is the density prediction using the adjoint operator of $\cP_t$. Given an initial density $\rho_0$, we  aim to predict the density at a future time at time $t>0$:
 \begin{align}
 \label{eqn:short_predict_forward_eqn}
     \rho_{t} = \cP_{t}^* \rho_0.
 \end{align}
where $\cP_{t}^*$ is the adjoint operator of $\cP_{t}$. Density prediction problem is equivalent to solving the Fokker-Planck equation 
\begin{align}\label{fpe}
    \frac{\partial \rho_t}{\partial t} = \cL^* \rho_t.
\end{align}
with initial condition $\rho_0.$

The implementation closely resembles the procedure for moment prediction, with only a minor modification to account for the adjoint operator. We express the density $\rho_t$ using the cluster bases:
\begin{equation*}
    \rho_{t} (x)= \mu(x)  \cdot \sum_{j=1}^{(dn)^2} c^t_{j} \psi_{j}(x),
\end{equation*}
where $\mu(x)$ is a mean-field approximation, 
and determine the coefficients by testing against $\psi_{j'}$:
\begin{displaymath}
    \langle \rho_t , \psi_{j'} \rangle =   \langle \cP_t^* \rho_0 ,  \psi_{j'}  \rangle = \langle  \rho_0 , \cP_t \psi_{j'}  \rangle. 
\end{displaymath}
Substituting the expansion of $\rho_t$ yields
\begin{align*}
 \sum_{j=1}^{(dn)^2} \langle  \psi_{j'},  \psi_{j}  \rangle_\mu  c^t_{j} 
  =  \sum_{j=1}^{(dn)^2}  \langle \cP_t \psi_{j'},   \psi_{j}  \rangle_\mu c^0_{j}.
\end{align*}
This can be written in matrix form as
\begin{align}
\label{eqn:forward_pred_linear}
  M_\mu^{0} c^t  = (M_{\mu}^{t})^* \cdot c^0
\end{align}
where $(M_\mu^{t})^*$ is the conjugate transpose of $M_\mu^{t}$. To compute the coefficients efficiently, we again approximate $M_{\mu}^{t}$ by its compressed representation $\tilde M_{\mu}^{t}$ according to \Cref{alg:two_level}, hence the density coefficients at time $t$  can be obtained as 
\begin{displaymath}
    c^{t} = \left(  M_\mu^{0}\right) ^{-1} \cdot \left((\tilde M_{\mu}^{t})^* \cdot c^{0}\right)
\end{displaymath}

\begin{remark}
For all prediction  problems considered in this work (see numerical examples in~ \Cref{sec:numerical_results}), we choose $\mu$ as a mean-field density defined in \eqref{mean-field}. Under such choice, the application of $ \left(  M_\mu^{0}\right) ^{-1}$ can be efficiently realized by constructing univariate basis functions  $\{\phi^i_j\}_{j=1}^n$ in each dimension $i$ via a Gram–Schmidt procedure to satisfy the orthogonality condition $\langle \phi_i^{i} , \phi^i_{j'}  \rangle_{\mu} = \delta_{j,j'}$. This construction yields a sparse Gram matrix $ M_\mu^{0}$, allowing its inverse to be applied in $\mathcal{O}(d^2)$ computational cost.
\end{remark}

\subsection{Boundary value problem: Solving committor function}
\label{sec:committor}
In studies of chemical kinetics and molecular dynamics, an important task is to quantify the probability of transition between two metastable states. One useful function is the \emph{committor function}~\cite{10.1063/1.2013256,E2010TransitionpathTA,E2005242,E2006}, which characterizes the probability that a system initialized at state $x$ will reach one metastable region before another. Formally, we consider transitions between two simply connected domains  $A,B\subset \mathbb{R}^d$ with smooth boundaries. 
The committor function $q: \Omega \setminus (A\cup B) \to [0,1]$ is defined as 
\begin{align}
    q(x)=\mathbb{P}(\tau_B<\tau_A | X_0=x)
\end{align}
where $\tau_A$ and $\tau_B$ are the stopping times \eqref{stop time} of sets $A$ and $B$. It represents the probability that a stochastic trajectory starting at $x$ will reach $B$ before $A$. The committor function solves the boundary value problem 
\begin{align}
\label{eqn:committor}
    \begin{cases}
        \cL q = 0  & \text{ in } \Omega \backslash (A\cup B) , \\ 
        q = 0     & \text{ on } \partial A,\\ 
        q=1       & \text{ on } \partial B. 
    \end{cases}
\end{align}

Using the stopped transition operator defined in~\eqref{eqn:semigroup_stop}, we can equivalently solve the fixed-point equation~\cite{pmlr-v145-li22a}
\begin{align}
\label{eqn:commitor_eigen}
   \cP_{t\wedge \tau_{A\cup B}} q = q  
\end{align}
for chosen $t$ with the boundary conditions in \eqref{eqn:committor}. To solve it, we approximate the solution as 
\begin{align}
\label{eqn:cluster_committor}
    q(x) = \sum_{j=1}^{(dn)^2} c_{j} \psi_{j}(x) + \mathbf{1}_B(x),
\end{align}
where the basis functions satisfy the boundary condition $\psi_{j}(x) = 0$ if $x \in A \cup B$. This can be achieved by defining $\psi_{j} =  \mathbf{1}_{\Omega \backslash (A\cup B)} \cdot \tilde{\psi}_{j}$ with $\tilde{\psi}_{j}$ being the standard cluster basis functions introduced in~\eqref{two cluster linear}. Substituting~\eqref{eqn:cluster_committor} into the weak formulation under the equilibrium density,
\begin{align}
\label{eqn:committor_galerkin}
\langle  \cP_{t \wedge \tau_{A\cup B} } q ,\psi_{j'}\rangle_{\rho_{\infty}} =  \langle  q , \psi_{j'} \rangle_{\rho_{\infty}}
\end{align}
 yields
$$
 \sum_{j=1}^{(dn)^2}  \langle \cP_{t\wedge \tau_{A\cup B}} \psi_{j},   \psi_{j'}  \rangle_{\rho_\infty} c_{j} +   \langle \cP_{t\wedge \tau_{A\cup B}} \mathbf{1}_B,   \psi_{j'}\rangle_{\rho_\infty}  =    \sum_{j=1}^{(dn)^2}  \langle\psi_{j},   \psi_{j'}  \rangle_{\rho_\infty} c_{j} +  \langle  \mathbf{1}_B,   \psi_{j'}\rangle_{\rho_\infty}. 
$$
Since $\psi_{j'}(x) = 0$ for any $x \in B$,  $\langle  \mathbf{1}_B,   \psi_{j'}\rangle_{\rho_\infty}$ vanishes. The equation then becomes the linear system
\begin{equation}\label{committor linear system}
    ( M_{\rho_{\infty}}^0- M_{\rho_{\infty}}^{t\wedge \tau_{A\cup B}}) c = f 
\end{equation}
where vector $f \in \R^{(dn)^2\times 1}$ is formulated elemenetwisely as 
\begin{displaymath}
     f_j =   \langle \cP_{t\wedge \tau_{A\cup B}} \mathbf{1}_B, \psi_{j}\rangle_{\rho_\infty}.
\end{displaymath}
At this point, we apply the proposed method to compress $( M_{\rho_{\infty}}^0- M_{\rho_{\infty}}^{t\wedge \tau_{A\cup B}})$ into $\tilde M_\star$ stored as the sparse and low-rank factors, and approximate $f$ by Monte Carlo integration   
\begin{displaymath}
    \tilde f_j = \frac{1}{N_{\text{src}}}\frac{1}{N_{\text{traj}}} \sum_{i:x^{(i) }\notin A\cup B}\psi_{j}(x^{(i)})N_B^{(i)}
\end{displaymath}
where $N_B^{(i)}$ be the number of trajectories $y^{(i,l)}$ that are eventually absorbed by $B$. The  coefficients of committor function are then computed as $c=\tilde M_\star^\dag \tilde{f}$. We approximate $\tilde M_\star^\dag$ by inverting an approximated compact SVD of $\tilde M_\star$, which can be obtained cheaply using randomized SVD~\cite{doi:10.1137/090771806}, since the compressed representation of $\tilde M_\star$  allows fast application of  $\tilde M_\star$ to any matrix. 

\begin{remark}
    To further stabilize the inversion and enforce boundary conditions explicitly,
we introduce additional constraints by sampling boundary points $x_{A}^{(i)} \in \partial A$ and $x_{B}^{(i)} \in \partial B$, $i=1,\ldots, N_{bc}$, and requiring
\begin{displaymath}
    \sum_{j=1}^{(dn)^2} c_{j} \psi_{j}(x_{A}^{(i)}) = 0,   \qquad  
      \sum_{j=1}^{(dn)^2} c_{j} \psi_{j}(x_{B}^{(i)}) = 1.
\end{displaymath}
Let $\Psi_{\partial A}, \Psi_{\partial B} \in \mathbb{R}^{N_{bc} \times (dn)^2}$ with $\Psi_{\partial A}(i,j)=\psi_{j}(x_{A}^{(i)})$ and $\Psi_{\partial B}(i,j)=\psi_{j}(x_{B}^{(i)})$. Then the augmented linear system is
\begin{align}
\label{eqn:linear_sys}
    \begin{pmatrix}
         \tilde M_\star  \\ 
        \Psi_{\partial A} \\ 
        \Psi_{\partial B}
    \end{pmatrix}
    c 
    =
    \begin{pmatrix}
         \tilde f \\
        \vec 0\\
        \vec 1
    \end{pmatrix}
\end{align}
where $\vec 0$ and $ \vec 1$ are respectively $N_{bc}\times 1$ all-zero and all-one vector. This explicit enforcement of boundary conditions enhances numerical stability and ensures accurate approximation of the committor function near the domain boundaries.
\end{remark}

\section{Theoretical results}
\label{sec:theory}
 In this section, we establish the existence of the proposed two-level compressed representation, providing the theoretical foundation that underlies the construction in~\Cref{alg:two_level}.  For simplicity, we present the proofs under the assumption of a reversible finite-state continuous-time Markov process with $\mu$ chosen to be the equilibrium density~\eqref{eq_density}. The simplification to a finite-state model is justified by the fact that, under a coarse-grained discretization of the state space, a diffusion process can be approximated by a Markov jump process~\cite{doi:10.1137/0328056,Schtte1999ADA}. In addition, we assume the entries of transition moment matrix can be evaluated exactly without Monte Carlo method for simplicity. Including the Monte Carlo error will be considered as future work. 

The section is organized as follows. In Section~\ref{sec:finite-state-markov}, we review several useful properties of the considered Markov process with a particular focus on its generator and introduce some definitions and assumptions that will be used throughout this section. Then in Section~\ref{sec:main results}, we explicitly construct a desired compressed representation and provide the corresponding error estimate.


\subsection{Reversible finite-state Markov process}
\label{sec:finite-state-markov}
We consider a discrete many-body Markov system defined on a lattice with $d$ sites. Let each site takes values in a  finite space denoted by $\cX$, and define the global configuration space $\Omega=\cX^d$. A configuration is then written as $x=(x_1,\ldots,x_d)$ with $x_i\in \cX$. We also denote 
$$
x_{-i}:=(x_1,\ldots,x_{i-1},x_{i+1},\ldots,x_d) \in \cX^{d-1}.
$$  
The reversible continuous-time Markov process $\{X_t\}_{t\geqslant 0}$ is governed by generator $\cL = \sum_{i=1}^d \cL_i$, where the local generator takes the form
\begin{align}
\label{def:generator2} 
    (\cL_i f)(x) = \sum_{\alpha \in \cX}  r_{i,\alpha}(x) \Delta_{i,\alpha} f(x) .
\end{align}
Here, the non-negative real number $r_{i,\alpha}$ is the local transition rate at site $i$, and $\Delta_{i,\alpha}$ is a linear operator defined as 
\begin{equation}\label{def:Delta}
    \Delta_{i,\alpha} f(x):= f(\tau_{i,\alpha} x) - f(x)
\end{equation}
with single-site update operator 
\begin{displaymath}
    \tau_{i,\alpha} x:= (x_1,\cdots,x_{i-1},\alpha,x_{i+1},\cdots,x_d)
\end{displaymath}
which replaces the state at site $i$ by $\alpha \in \cX$. 
For later use, we denote  
\begin{equation}\label{rmax and rbar}
    r_{\max} := \max_{i\in [d],\alpha\in \cX,x\in \Omega} r_{i,\alpha}(x) \quad \text{and} \quad          \bar{r} := \max_{x\in \Omega} \sum_{i=1}^d\sum_{\alpha\in \cX}  r_{i,\alpha}(x).
\end{equation}
Since we assume the Markov process is reversible, there exists an equilibrium density $\rho_\infty$ that satisfies the {\it detailed balance condition}, which under this notation can be written as 
\begin{equation}\label{eq:detailed_balance}
    \rho_\infty(x) r_{i,\alpha}(x) 
    = \rho_\infty(\tau_{i,\alpha}x) r_{i,x_i}(\tau_{i,\alpha}x) .
\end{equation}

We assume that the system has a  {\it finite-range} interaction,  meaning that the local transition rate $r_{i,\alpha}$ in~\eqref{def:generator2} depends only on the configuration of the system within a finite neighborhood around site $i$. This property can be formally expressed using the single-site update operator:
\begin{assumption}[Finite-range interaction]
\label{assum:finite range}
For any configuration $x \in\Omega$, there exists an integer $\kappa>0$ such that
    \begin{equation} 
       r_{i,\alpha}(x) = r^*_{i,\alpha}(x_{\mathcal{N}(i,\kappa)}).
    \end{equation}
    for some function $r^*_{i,\alpha}$ which only relies on variables in $\mathcal{N}(i,\kappa)$. Consequently, if $|i-j| > \kappa$, 
    $$ r_{i,\alpha}(x) = r_{i,\alpha}(\tau_{j,\alpha'} x) \quad \text{~for any~} \alpha,\alpha' \in \cX .$$
\end{assumption}

We also make the following assumption on the spectrum of the generator $\cL$, which holds for most reversible Markov systems with finite-range interactions~\cite{10.1214/aop/1176993067,Caputo2004SpectralGap,10.1214/aop/1041903209,cmp/1104253633}.
\begin{assumption}[Spectral gap]
 \label{assum:spec gap}
    Let $\{(\lambda_k,w_k)\}_{k=0}^\infty$ be the eigenpairs of the operator $-\cL$, satisfying 
  \begin{displaymath}
      -\cL  w_k =  \lambda_k w_k
  \end{displaymath}  
    where the eigenfunctions are orthonormal under  the equilibrium measure, i.e. $\langle w_i,w_j \rangle_{\rho_\infty} = \delta_{ij}$. We assume that $w_0 \equiv 1$, and there exists $\lambda>0$ such that 
    \begin{align}
    \label{eqn:eigengap}
        0=\lambda_0<\lambda_1:= \lambda \leqslant \lambda_2 \leqslant \cdots \leqslant \lambda_k\leqslant \cdots. 
    \end{align}   
    Consequently, the transition operator $\cP_t=e^{t\cL}$ has the same eigenfunctions while the sequence of eigenvalues become
    \begin{displaymath}
       0 < \cdots \leqslant  e^{-\lambda_k t} \leqslant \cdots \leqslant  e^{-\lambda_2 t} \leqslant e^{-\lambda t} < e^{-\lambda_0 t} = 1.
    \end{displaymath}   
\end{assumption}

To quantify the sensitivity of a function to individual sites, we introduce the oscillation and its associated seminorm.

\begin{definition}
    Given any  function $f$ defined on $\Omega$, we define its oscillation at site $i$ as 
\begin{equation} \label{def:osc}
    \Osc_i(f) = \max_{x_{-i}\in \cX^{d-1}} \max_{z,w\in \cX} \left| f(z,x_{-i}) - f(w,x_{-i}) \right|. 
\end{equation}
The seminorm induced by the oscillation is then defined as
\begin{displaymath}
    \|f\|_{\osc} = \sum_{i=1}^d \Osc_i(f). 
\end{displaymath}
\end{definition}

Finally, we  assume that the cluster basis functions satisfy the following boundedness assumption, which is satisfied by a broad class of basis functions considered in this work including Fourier bases, radial basis functions, and polynomial bases defined on a bounded domain:
\begin{assumption}[Boundedness of cluster basis]
 \label{assum:cluster bound}
    The two-cluster basis functions are bounded in the following sense
\begin{displaymath}
   \max_{j\in [(dn)^2]}\{\|\psi_{j}\|_{L^2_{\rho_\infty}},\|\psi_{j}\|_{\mathrm{osc}} \}  \leqslant   \mathcal{C}_0
\end{displaymath}
for some $\mathcal{C}_0>0.$
\end{assumption}

\subsection{Main results}
\label{sec:main results}
In this subsection, we summarize the major theoretical results that lead to the existence of the compressed representation. For simplicity, we assume the transition moment matrix $M_{\rho_\infty}^t$ defined in \eqref{transition matrix}  can be exactly obtained. 

We first analyze the intrinsic low-rank property of the transition moment matrix $M_{\rho_\infty}^t$, which is guaranteed by the following lemma: 
\begin{lemma}
\label{lemma:global_lowrank}
    Under \Cref{assum:spec gap} and \Cref{assum:cluster bound}, for any target rank $r>0$, there exists a rank-$r$ matrix $H_r\in \R^{(dn)^2\times (dn)^2}$ such that  
    \begin{align}
    \label{eqn:rank_r_approx}
 \| M_{\rho_{\infty}}^t- H_r\|_2\leqslant \mathcal{C}_0^2 (dn)^2  e^{-\lambda_{r} t}
    \end{align}
where $\lambda_r$ given in \Cref{assum:spec gap} is the $r$-th eigenvalue of the considered generator, and $\mathcal{C}_0$ is given in \Cref{assum:cluster bound}.
\end{lemma}
\begin{proof}
    We expand each two-cluster basis function by the eigenfunctions:
    \begin{align}
\psi_{i}=\sum_{k=0}^\infty v_{i,k} w_k, \qquad v_{i,k}:= \langle \psi_{i}, w_k \rangle_{\rho_{\infty}}. 
\end{align}
Plugging them into \eqref{transition matrix}, we obtain  the  following spectral representation of the transition moment matrix
\begin{equation}
\label{eqn:sepctral_rep_moment}
 \begin{split}
    M^t_{\rho_{\infty}}(i,j)  = \ &  \sum_{k=0}^\infty \sum_{k'=0}^\infty v_{i,k} v_{j,k'} \langle w_k, \cP_t w_{k'} \rangle_{\rho_\infty} \\
      = \ &  \sum_{k=0}^\infty \sum_{k'=0}^\infty  v_{i,k} v_{j,k'} e^{-\lambda_{k'} t}\langle w_k, w_{k'} \rangle_{\rho_\infty} 
    =\sum_{k=0}^\infty  e^{-\lambda_k t} v_{i,k} v_{j,k} .
    \end{split}
\end{equation}
Let $H_r\in \R^{(dn)^2 \times (dn)^2}$ such that 
\begin{align}
    H_r(i,j)=\sum_{k=0}^{r-1} e^{-\lambda_r t} v_{i,k} v_{j,k} 
\end{align}
which is a matrix of rank $r$. 
By \Cref{assum:spec gap}, we have 
\begin{equation}
\label{M-H max bound}
 \begin{split}
|M_{\rho_{\infty}}^t(i,j)- H_r(i,j)|\leqslant \ & e^{-\lambda_r t} \sum_{k=r}^\infty |v_{i,k}||v_{j,k}|\\
\leqslant \ & e^{-\lambda_r t} \left(\sum_{k=r}^\infty |v_{i,k}|^2\right )^{1/2} \left(\sum_{k=r}^\infty |v_{j,k}|^2\right )^{1/2} \\
\leqslant \ &  e^{-\lambda_r t} \left(\sum_{k=0}^\infty |v_{i,k}|^2\right )^{1/2} \left(\sum_{k=0}^\infty |v_{j,k}|^2\right )^{1/2}
=  e^{-\lambda_r t} \|\psi_{i}\|_{L^2_{\rho_\infty}}\|\psi_{j}\|_{L^2_{\rho_\infty}}\leqslant \mathcal{C}_0^2 e^{-\lambda_r t}. 
\end{split}
\end{equation}
Therefore equation~\eqref{eqn:rank_r_approx} easily follows by $\| M_{\rho_{\infty}}^t- H_r\|_2 \leqslant \| M_{\rho_{\infty}}^t- H_r\|_F \leqslant (dn)^2 \times \eqref{M-H max bound}$. 
\end{proof}
\Cref{lemma:global_lowrank} shows that as the lag time $t$ increases, 
the transition moment matrix $M_{\rho_\infty}^t$ can be approximated by a rank-$r$ matrix exponentially fast, 
at the rate of $e^{-\lambda_r t}$. Such low-rank property ensures that $M_{\rho_\infty}^t$ can be well approximated by CUR approximation introduced in the first step of \Cref{alg:two_level}. In the subsequent lemma, we provide an error bound for approximating a symmetric matrix by its own columns, which is a specialization of the result in \cite{GOREINOV19971} or \cite[p. 23-25]{Kressner} for general matrices. 
\begin{lemma}[\cite{Kressner}, p. 23-25]
\label{lemma:CUR bound}
Given a symmetric matrix $A\in \R^{m\times m}$, there exists an index set $J \subset [m]$ of size $|J| = r$ and a matrix $U \in \R^{r\times r}$ such that  
    \begin{equation*}
        \|A - A(:,J) U A(:,J)^T \|_2 \leqslant \sigma_{r+1}(A) (4\sqrt{rm}+1).
    \end{equation*}
    Furthermore, if $r<\text{rank}(A)$, we have  
    \begin{displaymath}
        \|U\|_2 \leqslant \frac{\sqrt{rm}}{\sigma_{r+1}(A)}.
    \end{displaymath}
\end{lemma}

Next, we provide a theoretical justification corresponding to the second step of \Cref{alg:two_level}. The approximation in~\eqref{eqn:G_approx} is supported by the following decay-of-correlation property:
\begin{lemma}[Decay of correlation]
\label{lemma:decay of correlation}
Under \Cref{assum:finite range} and \ref{assum:spec gap},  given two-cluster basis functions $\psi_{i}$ and $\psi_{j}$ for $i,j\in [(dn)^2]$, for any lag time $t>0$, we have the following uniform upper bound which does not rely on $t$:
    \begin{equation}
        | \langle \psi_{i}, \cP_{t} \psi_{j} \rangle_{\rho_\infty} - \langle \psi_{i} \rangle_{\rho_\infty}  \langle  \psi_{j} \rangle_{\rho_\infty} | \leqslant \mathcal{C}_1 d e^{-\frac{\lambda \xi(i,j) }{v+\lambda \kappa}} \|\psi_{i}\|_{\osc} \|\psi_{j}\|_{\osc}. 
    \end{equation}
    Here $\xi(i,j)$ is the distance between two sites defined in \eqref{site distance}, $\mathcal{C}_1 = |\cX| r_{\max}\max(\frac{1}{\lambda},\frac{\kappa}{v})$, $\kappa$ is the range of interaction by \Cref{assum:finite range} and $ v = 2(e - 1)(2\kappa+1)\kappa ( |\cX| r_{\max} + \bar{r})$, where $r_{\max}$ and $\bar r$ are given in \eqref{rmax and rbar}.
\end{lemma}
The proof of this lemma is deferred to~\Cref{sec:proof_of_decay_of_correlation}. The result above indicates that for each entry of the transition moment matrix \eqref{transition matrix} formed by two-cluster bases $\{\psi_{i}\}_{i\in [(dn)^2]}$, if the site indices $s(i)$ and $s(j)$ are far away, the value of this entry can be approximated by the product of first-order moments. With \Cref{lemma:decay of correlation}, one can approximate each selected column or row in the transition moment matrix \eqref{transition matrix} as a sum of sparse factor and low-rank factor as implemented in Step 2 of \Cref{alg:two_level}. The following lemma provides one such approximation and the corresponding error analysis:

\begin{lemma}
\label{lemma:column_approx}
Under Assumptions~\ref{assum:finite range},~\ref{assum:spec gap} and~\ref{assum:cluster bound},  consider the transition moment matrix $M_{\rho_\infty}^t\in \R^{(dn)^2\times (dn)^2}$ defined in~\eqref{transition matrix}. For any $j\in [(dn)^2]$ and a prescribed bandwidth $\delta>0$, $j$-th column of $M_{\rho_\infty}^t$ can be approximated by $G_j$ formulated as 
\begin{equation}
\label{eqn:sparse+rank-one}
 G_j  = P_j+\mathrm{vec}(W_j),
\end{equation}
where $P_j$ is a sparse vector with $\cO(dn\delta)$ nonzero entries defined in~\eqref{eqn:G_sparse}, and $W_j$ is a $dn\times dn$ rank-one matrix. The approximation error is given by 
\begin{align}
\label{eqn:col_error}
    \|M_{\rho_\infty}^t(:,j)-G_j\|_2\leqslant 2\mathcal{C}_2 nd\sqrt{d(d-\delta)}e^{-\frac{\lambda \delta}{v+\lambda \kappa}}
\end{align}
where the constant $\mathcal{C}_2=\mathcal{C}_1\max\{\mathcal{C}_0^2,\mathcal{C}_0^3\}.$
\end{lemma}
\begin{proof}
We define $G_j$ by setting $P_j$ as \eqref{eqn:G_sparse} and 
\begin{displaymath}
  W_j((k_1,l_1),(k_2,l_2)) := \langle \phi^{k_1}_{l_1} \rangle_{\rho_\infty} \cdot \langle \phi^{k_2}_{l_2} \rangle_{\rho_\infty} \cdot \langle \cP_t \psi_j \rangle_{\rho_\infty}, \quad \forall (k_1,l_2),(k_2,l_2)\in [d]\times [n]
\end{displaymath}
which is the rank-one matrix defined in~\eqref{eqn:rank-one-mat}.
For any $i$ such that $\min\{\xi(i,j), |s_1(i)-s_2(i)| \}> \delta$, from the triangular inequality and the definition of $G_j$, we obtain 
    \begin{equation}
            |M_{\rho_\infty}^t(i,j)- G_j(i)|= |\langle \phi_{b_1(i)}^{s_1(i)}\phi_{b_2(i)}^{s_2(i)}, \cP_t \psi_j \rangle_{\rho_\infty}-\langle \phi_{b_1(i)}^{s_1(i)}\rangle_{\rho_\infty}
        \langle \phi_{b_2(i)}^{s_2(i)}\rangle_{\rho_\infty}
        \langle  \psi_j \rangle_{\rho_\infty}|
        \leqslant T_1+T_2
    \end{equation}
    where 
    \begin{equation}
        \begin{split}
            T_1  =  \ &|\langle \phi_{b_1(i)}^{s_1(i)},\phi_{b_2(i)}^{s_2(i)}, \cP_t \psi_j \rangle_{\rho_\infty}-\langle \phi_{b_1(i)}^{s_1(i)},\phi_{b_2(i)}^{s_2(i)}\rangle_{\rho_\infty} \langle  \psi_j \rangle_{\rho_\infty} |,\\
            T_2  = \ & |\langle \phi_{b_1(i)}^{s_1(i)},\phi_{b_2(i)}^{s_2(i)}\rangle_{\rho_\infty}  -\langle \phi_{b_1(i)}^{s_1(i)}\rangle_{\rho_\infty}
        \langle \phi_{b_2(i)}^{s_2(i)}\rangle_{\rho_\infty}
        |\cdot |\langle  \psi_j \rangle_{\rho_\infty}|.
        \end{split}
    \end{equation}
    \Cref{lemma:decay of correlation} directly implies 
    $$
    T_1 \leqslant \mathcal{C}_1 \mathcal{C}_0^2d e^{-\frac{\lambda \delta}{v+\lambda \kappa}}.
    $$
    Applying \Cref{lemma:decay of correlation} again to two-cluster basis functions $1\cdot \phi_{b_1(i)}^{s_1(i)}$ and $1\cdot \phi_{b_2(i)}^{s_2(i)}$, we have 
    $$
    T_2\leqslant  \mathcal{C}_1 \mathcal{C}_0^2d e^{-\frac{\lambda |s_1(i)-s_2(i)|}{v+\lambda \kappa}}|\langle  \psi_j \rangle_{\rho_\infty}|\leqslant  \mathcal{C}_1 \mathcal{C}_0^3d e^{-\frac{\lambda \delta}{v+\lambda \kappa}}.
    $$
    Therefore,  for $\min\{\xi(i,j), |s_1(i)-s_2(i)| \}> \delta$, 
    \begin{align}
    \label{eqn:non-zero_err}
        |M_{\rho_\infty}^t(i,j)- G_j(i)|\leqslant 2 \mathcal{C}_1  \max\{\mathcal{C}_0^2,\mathcal{C}_0^3\}d e^{-\frac{\lambda \delta}{v+\lambda \kappa}}= 2 \mathcal{C}_2d e^{-\frac{\lambda \delta}{v+\lambda \kappa}},
    \end{align}
    and otherwise $|M_{\rho_\infty}^t(i,j)- G_j(i)|=0$. Counting at most $n^2d(d-\delta)$ such entries having  nonzero error as in~\eqref{eqn:non-zero_err}, we reach to the bound \eqref{eqn:col_error}. 
\end{proof}

Finally, we  provide the following overall error estimate for the compressed representation.
\begin{theorem}
\label{thm:final_bd}
Under ~\Cref{assum:finite range}, \ref{assum:spec gap} and~\ref{assum:cluster bound},  consider the transition moment matrix $M_{\rho_\infty}^t\in \R^{(dn)^2\times (dn)^2}$ defined in~\eqref{transition matrix}. For bandwidth $\delta >0$, $M_{\rho_\infty}^t$ can be approximated by the rank-$r$ factorization $C U C^{T}$, where $U\in \R^{r\times r}$, $C$ is an approximation of $M_{\rho_\infty}^t(:,J)$ for some column indices $J\subset [(dn)^2]$, and each column of $C$ is the vectorization of the sum of a $\cO(nd\delta)$-sparse matrix  and a rank-one matrix. Furthermore, if 
\begin{align}
\label{eqn:N_spa}
    \delta \;\geqslant\; 
    \frac{v + \lambda \kappa}{\lambda}\,
    \log \!\left(
    \frac{2\mathcal{C}_2 d^2n^2\sqrt{r}}
         {\sigma_1(M_{\rho_\infty}^t)}
    \right),
\end{align}
the following relative error bound holds:
\begin{equation}\label{eqn:final_bd}
\frac{\| M_{\rho_\infty}^t - CUC^T\|_2}
     {\| M_{\rho_\infty}^t \|_2}
\;\leqslant\;
\frac{\mathcal{C}_0^2 (dn)^2 (1 + 4\sqrt{r}\,dn)   e^{-\lambda_{r} t}}
     {\sigma_{1}(M_{\rho_\infty}^t)}
\;+\;
\frac{
6 \mathcal{C}_2  (dn)^2 r\sqrt{d(d-\delta)} e^{-\frac{\lambda \delta}{v+\lambda \kappa }}
}{
\sigma_{r+1}(M_{\rho_\infty}^t)
},
\end{equation}
where 
\begin{itemize}
    \item $\kappa>0$ is the interaction range of the $\cL$ in Assumption~\ref{assum:finite range};
     \item $\lambda,\lambda_r>0$ are eigenvalues of $\mathcal{L}$ defined in \Cref{assum:spec gap};
        \item $\mathcal{C}_0$ is the upper bound in~\Cref{assum:cluster bound};
     \item $v$ is the constant appearing in \Cref{lemma:decay of correlation};
    \item $\mathcal{C}_2$ is the constant defined in~\Cref{lemma:column_approx}.
\end{itemize}
\end{theorem}
Before proving the theorem, we make some remarks on the meaning of this error bound. 
\begin{itemize}
    \item[(1)] The first term in the right-hand side of \eqref{eqn:final_bd} arises from approximating $M_{\rho_\infty}^t$ by a rank-$r$  factorization, where the factors are selected columns of $M_{\rho_\infty}^t$.
This approximation error decays exponentially with the $r$-th eigenvalue of the Markov semigroup, namely  $e^{-\lambda_r t}$, and  reflects the intrinsic  low-rankness of $M_{\rho_\infty}^t$.
    \item[(2)] The second term in ~\eqref{eqn:final_bd} corresponds to replacing each selected column by the representation in~\eqref{eqn:sparse+rank-one}.  
This approximation is controlled by the decay of correlation established in~\Cref{lemma:decay of correlation}, leading to exponential decay in the  bandwidth~$\delta$.  
The magnitude of this term is further amplified by the conditioning of the rank-$r$ factorization, characterized by $\sigma_{r+1}(M_{\rho_\infty}^t)^{-1}$.
\end{itemize}

\begin{proof}[Proof of~\Cref{thm:final_bd}]
    Since $M^t_{\rho_\infty} \in \R^{(dn)^2 \times (dn)^2}$ is PSD, by \Cref{lemma:CUR bound}, there exist an index set $J \subset [(dn)^2]$ and a matrix $U\in \R^{r\times r}$, such that 
    \begin{equation}
    \label{eq:CA} 
        \|M_{\rho_\infty}^t - M_{\rho_\infty}^t(:,J) U  M_{\rho_\infty}^t(:,J)^T\|_2
    \leqslant (1+4\sqrt{r}\,dn)\,\sigma_{r+1}(M_{\rho_\infty}^t)
    \end{equation}
    and 
\begin{equation}\label{eqn:assump_tail_sigvals}
   \|U\|_2 
   \leqslant \frac{dn\sqrt{r}}{\sigma_{r+1}(M_{\rho_\infty}^t)}.
\end{equation}
By \Cref{lemma:global_lowrank}, 
\begin{equation}\label{eq:CA2}
    \sigma_{r+1}(M_{\rho_\infty}^t)  \leqslant 
      \min_{A: \ \text{rank}(A)\leqslant r} \|M_{\rho_\infty}^t-A\|_F
     \leqslant  \|M_{\rho_\infty}^t-H_r\|_F  \leqslant \mathcal{C}_0^2 (dn)^2  e^{-\lambda_{r} t}.
\end{equation}
Combining \eqref{eq:CA} with \eqref{eq:CA2} yields
\begin{equation}\label{eqn:term1}
    \|M_{\rho_\infty}^t -M_{\rho_\infty}^t(:,J) U  M_{\rho_\infty}^t(:,J)^T \|_2
    \leqslant (1+4\sqrt{r}dn)\mathcal{C}_0^2 (dn)^2  e^{-\lambda_{r} t}.
\end{equation}
We define $C=(G_1,\ldots,G_r)\in \mathbb{R}^{(dn)^2\times r}$, where each column $G_l$ approximates  $M_{\rho_\infty}^t(:,j_l)$ for $l=1,\ldots,r$ and has the desired structure as stated in~\Cref{lemma:column_approx}.
Define the difference $E:= C-M_{\rho_\infty}^t(:,J)$,  through the triangular inequality, we can bound the difference between the two rank-$r$ approximation 
\begin{equation}
\begin{split}
    \| M_{\rho_\infty}^t(:,J) U  M_{\rho_\infty}^t(:,J)^T - CUC^T\|_2 
    & \leqslant (2\|E\|_2\|M_{\rho_\infty}^t(:,J)\|_2+\|E\|_2^2)\cdot \|U\|_2^2
    \\
    &\leqslant (2\|E\|_2\|M_{\rho_\infty}^t\|_2+\|E\|_2^2)\cdot \|U\|_2^2
\end{split}
\end{equation}
By~\Cref{lemma:column_approx}, we can bound the operator norm of the difference
\begin{equation}
    \|E\|_2\leqslant \|E\|_F=\sqrt{\sum_{l=1}^r \|M_{\rho_\infty}^t(:,j_l)- C(:,l)\|_2^2}=2\mathcal{C}_2 n d\sqrt{d(d-\delta)r}e^{-\frac{\lambda \delta}{v+\lambda \kappa}}.
\end{equation}
If we choose the bandwidth $\delta$ large enough such that \eqref{eqn:N_spa} is satisfied, one can check that $\|E\|_2 \leqslant \|M^t_{\rho_\infty}\|_2$. Hence, by \eqref{eqn:assump_tail_sigvals} 
\begin{equation}
 \label{eqn:term2}
 \begin{split}
     \| M_{\rho_\infty}^t(:,J) U  M_{\rho_\infty}^t(:,J)^T - CUC^T\|_2  & \leqslant 3\|E\|_2 \|M^t_{\rho_\infty}\|_2 \|U\|_2 \\
     &\leqslant 6 \mathcal{C}_2 d^2 n^2 r\sqrt{d(d-\delta)} e^{-\frac{\lambda \delta}{v+\lambda \kappa }} \cdot \frac{\sigma_{1}(M_{\rho_\infty}^t)}{\sigma_{r+1}(M_{\rho_\infty}^t)}.
      \end{split}
\end{equation}
Applying the triangle inequality to \eqref{eqn:term1} and \eqref{eqn:term2} gives the relative error in \eqref{eqn:final_bd}.

\end{proof}

\section{Numerical results}
\label{sec:numerical_results}
In this section, we provide numerical examples for solving high-dimensional problems in \Cref{sec:applications} using the proposed algorithm. While our methods can be applied to general Markov processes, as a representative case, we consider a $d-$dimensional overdamped Langevin dynamics throughout this section, which is governed by the stochastic differential equation
\begin{align}
\label{eqn:langevin_dynamics}
    \mathrm{d} X_t  = -\nabla V(X_t) \mathrm{d}t+\sqrt{2\beta^{-1}}\mathrm{d}W_t, \qquad
    X_0 = x_0
\end{align}
where $V:\R^d\to \R$ is a potential function and $\beta>0$ is the inverse temperature,  $W_t$ is a standard $d-$dimensional Brownian motion and $x_0 \in \R^d$ is the initial state. Unless stated otherwise, we will take $d=50$ in all experiments throughout this section to evaluate the performance of the proposed algorithm in high-dimensional settings.

We consider the following many-body potential defined on a hypercube, which is derived from Ginzburg-Landau model and is widely used in statistical physics \cite{hoffmann2012GL}:
\begin{equation}
 \label{gl1d_potential}
     V(x_1,\ldots,x_d)=
     \sum_{i=1}^{d+1} \frac{\gamma}{2} 
     \left( \frac{x_i-x_{i-1}}{h}\right)^2+\frac{1}{4\gamma} (1-x_i^2)^2 , \qquad x\in \Omega:=[-L,L]^d
 \end{equation}
where $x_0 = x_{d+1} = 0$, $h = 1/(1+d)$ and $\gamma$ is a positive parameter describing the coupling intensity of two nearest-neighbor sites. The two global minima $V(x)$, denoted by  $x_{\pm }$,  are shown in \Cref{fig:gl1d_minima}, where $\gamma$ is chosen to be $0.05.$

\begin{figure}[ht!]
\centering
\includegraphics[width=0.49\textwidth]{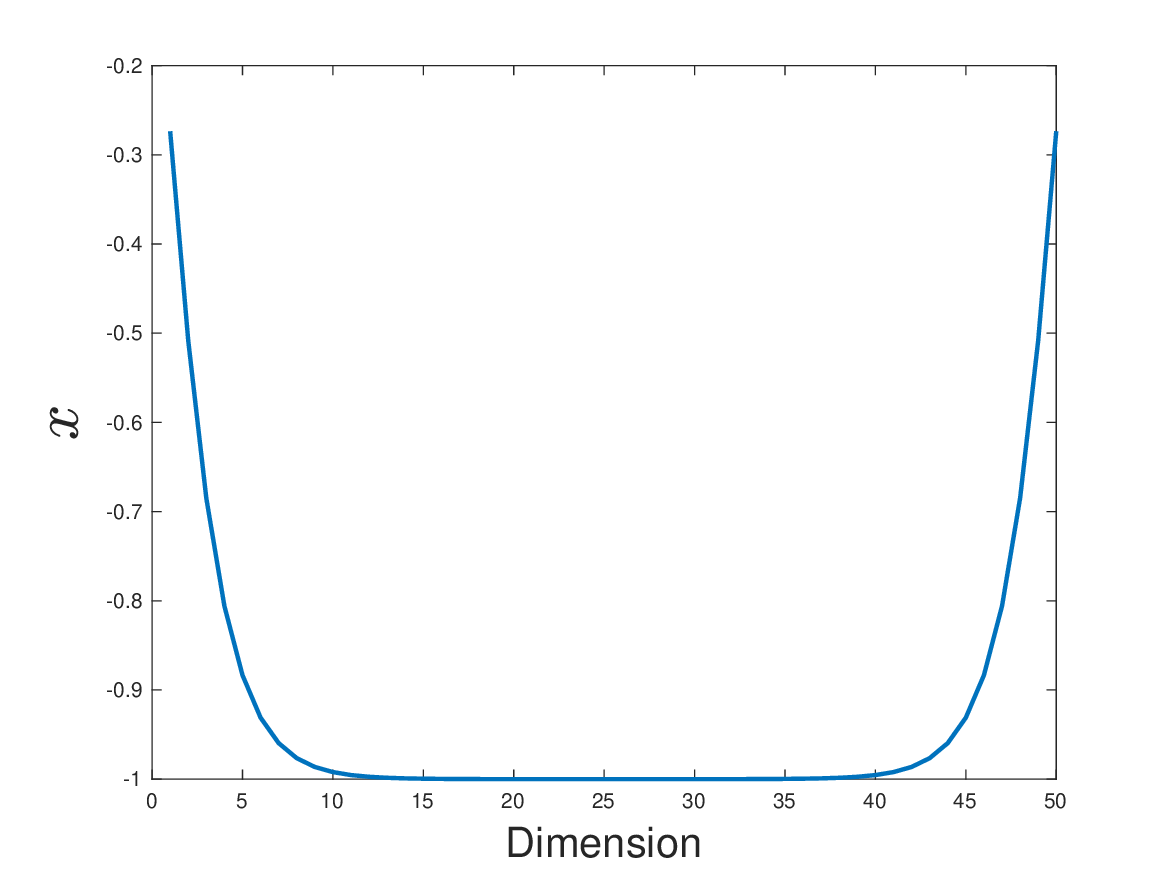}
\includegraphics[width=0.49\textwidth]{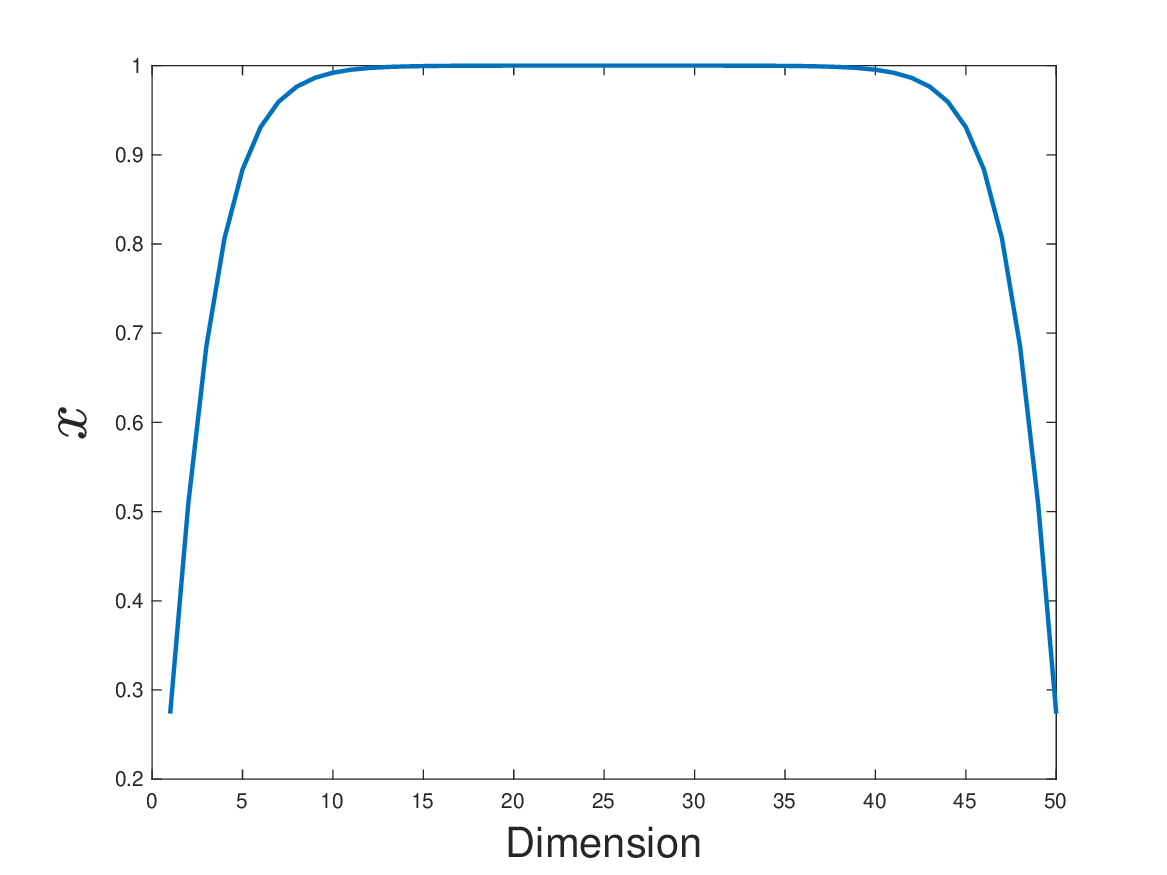}
    \caption{Global minimizer $x_-$ (left) and $x_+$ (right) for $V(x)$.}
\label{fig:gl1d_minima}
\end{figure}

The rest of this section is organized as follows. In~\Cref{sec:sanity_check}, we numerically verify the structure of the transition moment matrix that motivate our algorithms \Cref{alg:two_level}, and also validate some theoretical results in \Cref{sec:theory}. In~\Cref{sec:backward_results}, we demonstrate the performance of our method for moment prediction discussion in \Cref{sec:backward}. In~\Cref{sec:forward_results}, we apply the proposed method for density prediction
discussed in \Cref{sec:forward}. Finally in~\Cref{sec:committor_results}, we present results for solving the committor functions discussed in \Cref{sec:committor}.

\subsection{Structure of compressed transition moment matrix and approximation error}
\label{sec:sanity_check}
In this subsection, we demonstrate the structure of the transition moment matrix compressed by our proposed algorithm \Cref{alg:two_level}, and validate some theoretical results in \Cref{sec:theory} through several simple numerical experiments. 

In the first experiment, we examine the global low-rank structure of the transition moment matrix, which motivates the first-level compression described in~\Cref{sec:first_compression}. In particular, we aim to check if the transition moment matrix becomes more low-rank for larger lag time as justified theoretically by \Cref{lemma:global_lowrank}. We set the system dimension to $d=10$, coupling intensity to $\gamma=0.005$, and inverse temperature to $\beta=1/20$.  We use five univariate radius basis functions to form the two-cluster bases. The corresponding transition moment matrices are evaluated using the Monte Carlo method with $\mu=\rho_\infty$, at lag times $t=10^{-4},10^{-2},1$. We use $N_{\text{src}}=5000$ source samples and $N_{\text{traj}}=200$ trajectories for each source samples. In~\Cref{fig:check_global}, we show the $\log_{10}$-scaled singular values of the full transition moment matrices, after normalizing each set of singular values by its largest singular value. As expected, the singular values decay more rapidly for transition moment matrix with larger lag time.

\begin{figure}[ht!]
\centering 
\begin{tikzpicture}[scale=0.99]
    \node[inner sep=0] at (8,0) {\includegraphics[width=0.5\textwidth]{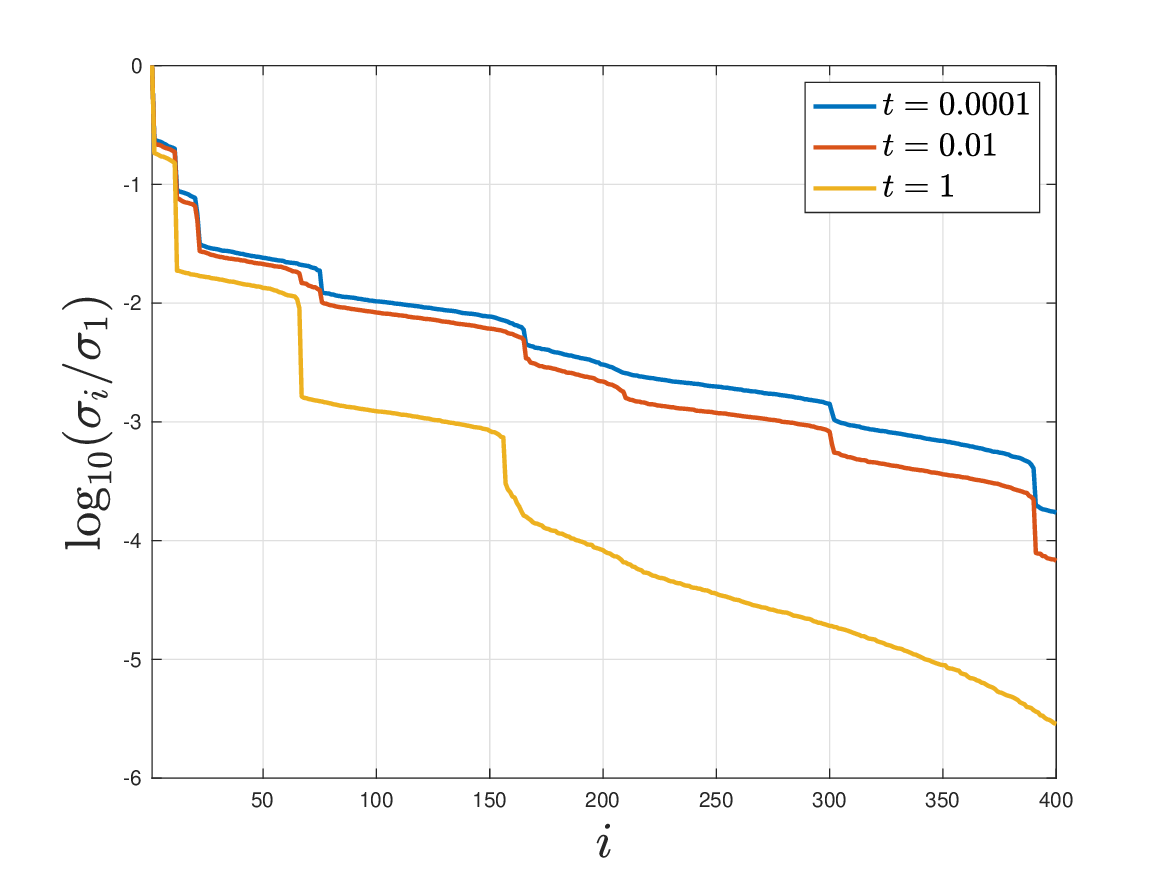}};
\end{tikzpicture}
\caption{The leading 400 $\log_{10}$-scaled singular values of the full transition moment matrices evaluated at lag times $t=0.0001,0.01,1$, after normalizing each set of singular values by the largest singular value.}
\label{fig:check_global}
\end{figure}

In the second experiment, we examine the decay-of-correlation property, which motivates the second-level compression described in~\Cref{sec:second_compression}. In particular, we aim to validate the exponential decay of correlations stated in \Cref{lemma:decay of correlation}. To this end, we use the equilibrium measure $\mu = \rho_\infty$ associated with potential \eqref{gl1d_potential}, and set the system dimension to $d=10$, coupling intensity to $\gamma=0.03$, inverse temperature to $\beta=1/10$. For simplicity, we consider one-cluster functions $g_i(x) = (\sin(x_i))^2$ for $i = 1,\cdots,d$. In general, evaluating the correlation
    \begin{equation}
     \label{correlation}
        | \langle g_{i}, \cP_{t} g_{j} \rangle_{\rho_\infty} - \langle g_{i} \rangle_{\rho_\infty}  \langle  g_{j} \rangle_{\rho_\infty} |  
    \end{equation}
requires Monte Carlo sampling to approximate both the conditional expectation $\cP_{t} g_{j}$ and $\rho_\infty$-inner product. This introduces non-negligible stochastic error, which makes it difficult to reliably verify an exponential decay rate. Therefore, we restrict to the case for lag time $t=0$, where $\cP_{t} g_{j} = g_{j}$. To compute correlations accurately and deterministically, we compress $\rho_\infty$ into the tensor-train format using the DMRG-cross algorithm~\cite{savostyanov2011fast}, which has demonstrated high accuracy for equilibrium densities under the 1D Ginzburg–Landau potential \eqref{gl1d_potential}~\cite{peng2025}. Leveraging the tensor-train structure of $\rho_\infty$, we evaluate \eqref{correlation} without Monte Carlo error. \Cref{fig:decay_of_cor} plots the $\log_{10}$-scaled correlation versus the site distance $|i - j|$ for the chosen test functions. As expected, the resulting curve exhibits an approximately straight, descending line on a semilogarithmic scale, confirming the exponential decay of spatial correlations predicted in~\Cref{lemma:decay of correlation}.

\begin{figure}[ht!]
    \centering
    \includegraphics[width=0.5\linewidth]{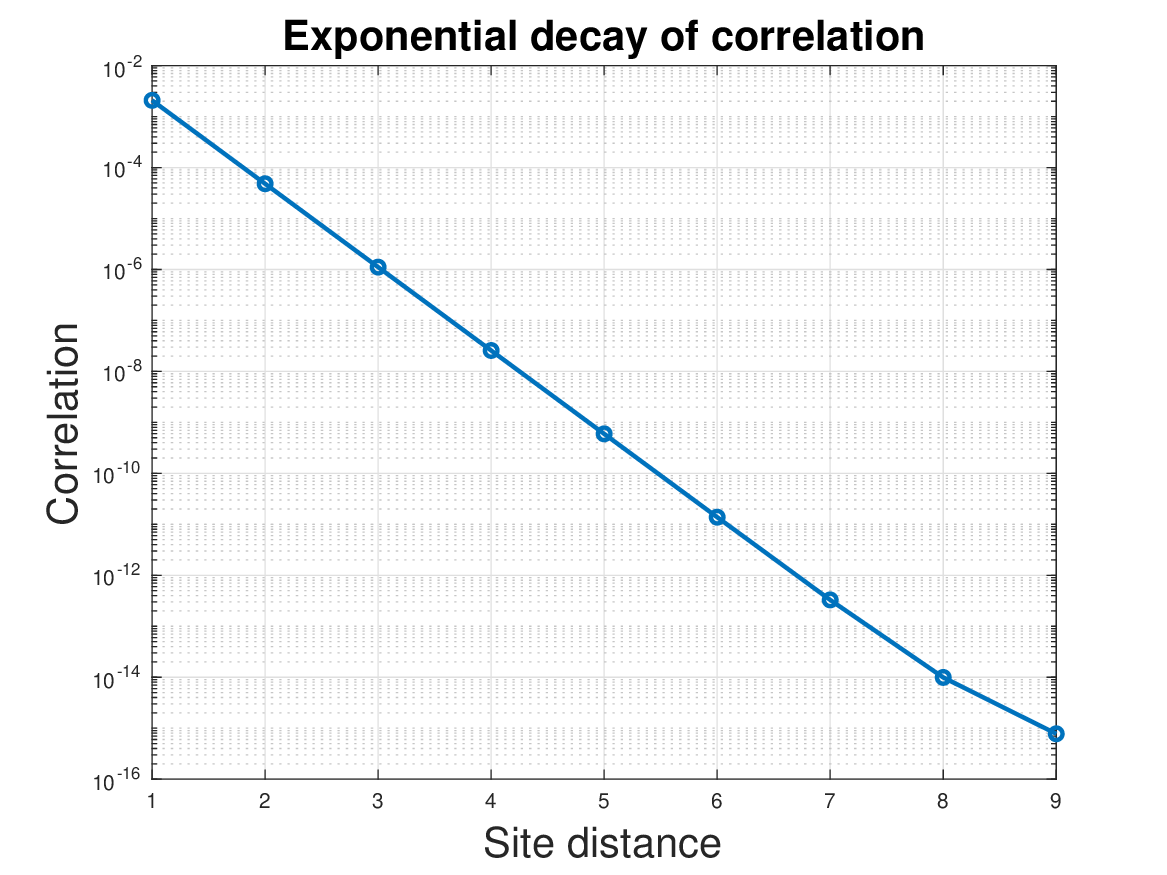}
    \caption{$\log_{10}$-scaled correlation $\lvert \langle g_i, g_j \rangle_{\rho_\infty} - \langle g_i \rangle_{\rho_\infty} \langle g_j \rangle_{\rho_\infty} \rvert$ versus the site distance $|i-j|$ for one-cluster functions $g_i(x) = (\sin(x_i))^2$.}
    \label{fig:decay_of_cor}
\end{figure}

Finally, we examine the sparse-plus-low-rank structure \eqref{col_sparse_lr_decomp} of a typical slice. We use the same parameters as in the first example, except that we increase the dimension to $d=50$ and use lag time $t=0.01$.  In \Cref{fig:check_slice}, we display the matricized slice corresponding to the column with indices $s_1=2,b_1=15,s_2=2,b_2=25$ of the transition moment matrix, along with its sparse component constructed using bandwidth $\delta=5$ (matricization of $P_j$ in \eqref{col_sparse_lr_decomp}), and the low-rank component (matricization of $Q_j$ in \eqref{col_sparse_lr_decomp}). As illustrated in \Cref{fig:check_slice}(B), the sparse component accurately captures the strongly correlated entries of the original slice in \Cref{fig:check_slice}(A). After removing this sparse part, the residual shown in \Cref{fig:check_slice}(C) exhibits clear low-rank structure. To substantiate this observation, \Cref{fig:check_slice}(D) plots the spectrum of the original slice and the residual, confirming that the residual factor is indeed numerically low-rank.     
\begin{figure}[ht!]
\centering
\begin{tikzpicture}[scale=0.99]
\node[inner sep=0] at (0,0)
{\includegraphics[width=0.8\textwidth,trim=100 200 50 200]{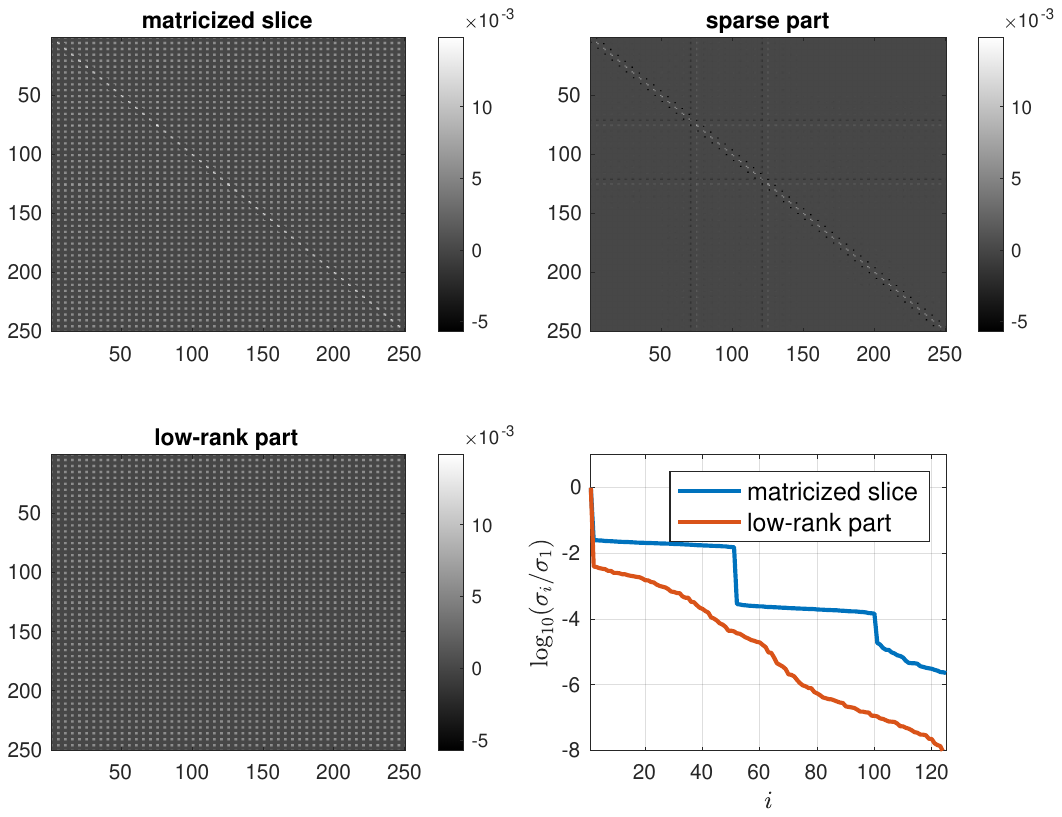}};
\node[inner sep=0] at (-3.5,0.25) {\footnotesize (A)};
\node[inner sep=0] at (3,0.25) {\footnotesize (B)};
\node[inner sep=0] at (-3.5,-5.5) {\footnotesize (C)};
\node[inner sep=0] at (3,-5.5) {\footnotesize (D)};
\end{tikzpicture}
\caption{(A) Matricization of a representative slice $M_{\rho_\infty}^t(:,j)$ of the transition moment matrix; 
(B) sparse component $P_j$; 
(C) low-rank residual $Q_j = M_{\rho_\infty}^t(:,j) - P_j$; 
(D) $\log_{10}$-scaled spectrum of the original slice $M_{\rho_\infty}^t(:,j)$ and the residual $Q_j$.}
\label{fig:check_slice}
\end{figure}

\subsection{Moment prediction}
\label{sec:backward_results}
Next, we apply  the proposed algorithm to solve the moment prediction problem, corresponding to the application discussed in \Cref{sec:backward}. 
In this example, we consider the following two functions:
\begin{equation*}
    f_+(x_1,\ldots,x_d) = \frac{1}{d}\sum_{i=1}^d (x_i - 1)^2, 
    \qquad  
    f_-(x_1,\ldots,x_d) = \frac{1}{d}\sum_{i=1}^d (x_i + 1)^2,
\end{equation*}
which approximately measure the mean squared Euclidean distance from the two global minima of the Ginzburg--Landau potential \(V(x)\) shown in \Cref{fig:gl1d_minima}, and the corresponding predictions $\cP_Tf_\pm$ for simulation time $T$ represent the expected squared distance of a stochastic trajectory from either well after a finite time.

In this experiment, we set $d=50$, coupling intensity $\gamma = 0.005$, inverse temperature $\beta = 5$ and consider the prediction $\cP_T f_{\pm}$ for $T=1$. In the algorithm, we set lag time $t = 0.2$, then the desired prediction $\cP_T f_{\pm}$ can be obtained by applying the learned Markovian transition moment matrix for 5 iterations. In addition, we use number of source samples $N_{\text{src}} = 2000$, number of trajectories for each source sample $N_{\text{traj}} = 100$, and use first 5 Legendre bases as the univariate bases for the cluster bases. 
In Step 1 of \Cref{alg:two_level}, we set the CUR rank $r_1=600$. In Step 2 of \Cref{alg:two_level}, we choose bandwidth $\delta=5$ for the sparse factor, and implement truncated SVD to compress the low-rank factor, with rank $r_2$ determined adaptively by the SVD truncation error $ 10^{-3}$.

\begin{figure}[ht!]
\centering
\includegraphics[width=0.49\textwidth]{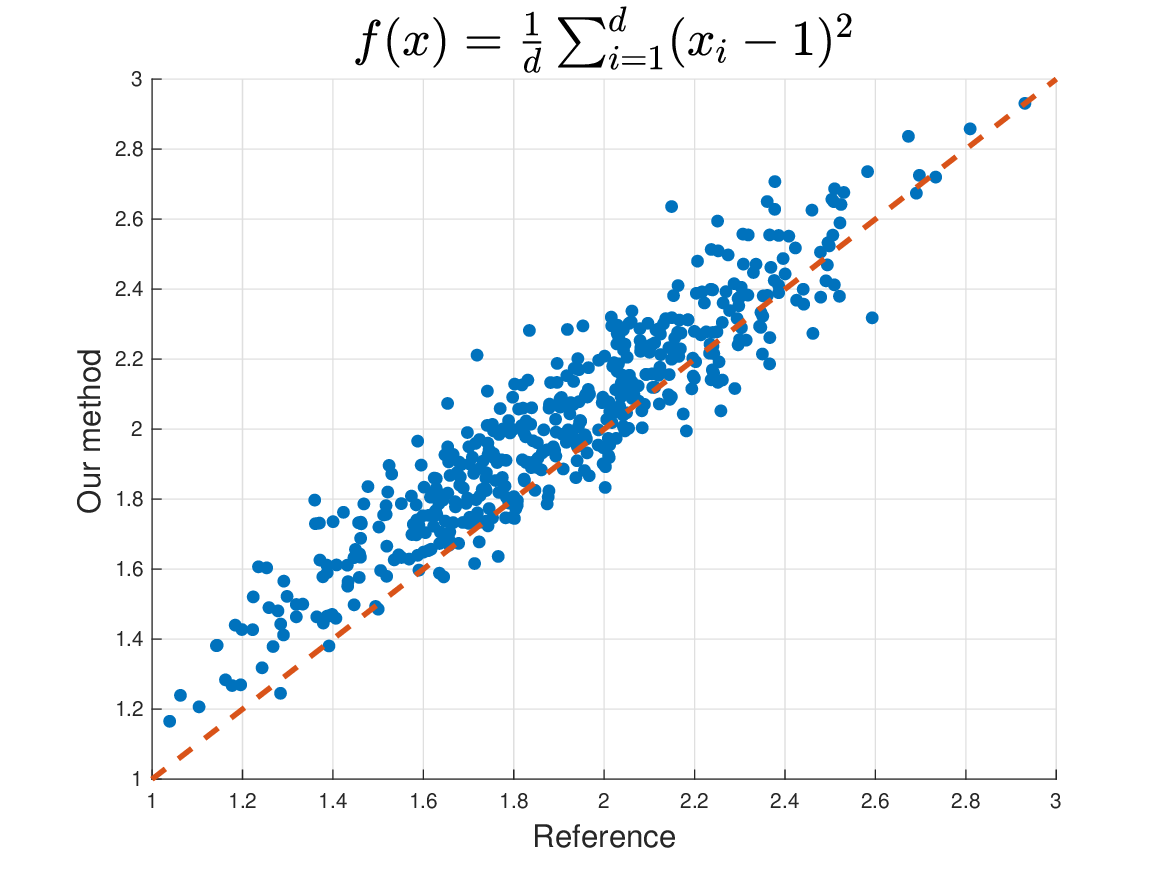}
\includegraphics[width=0.49\textwidth]{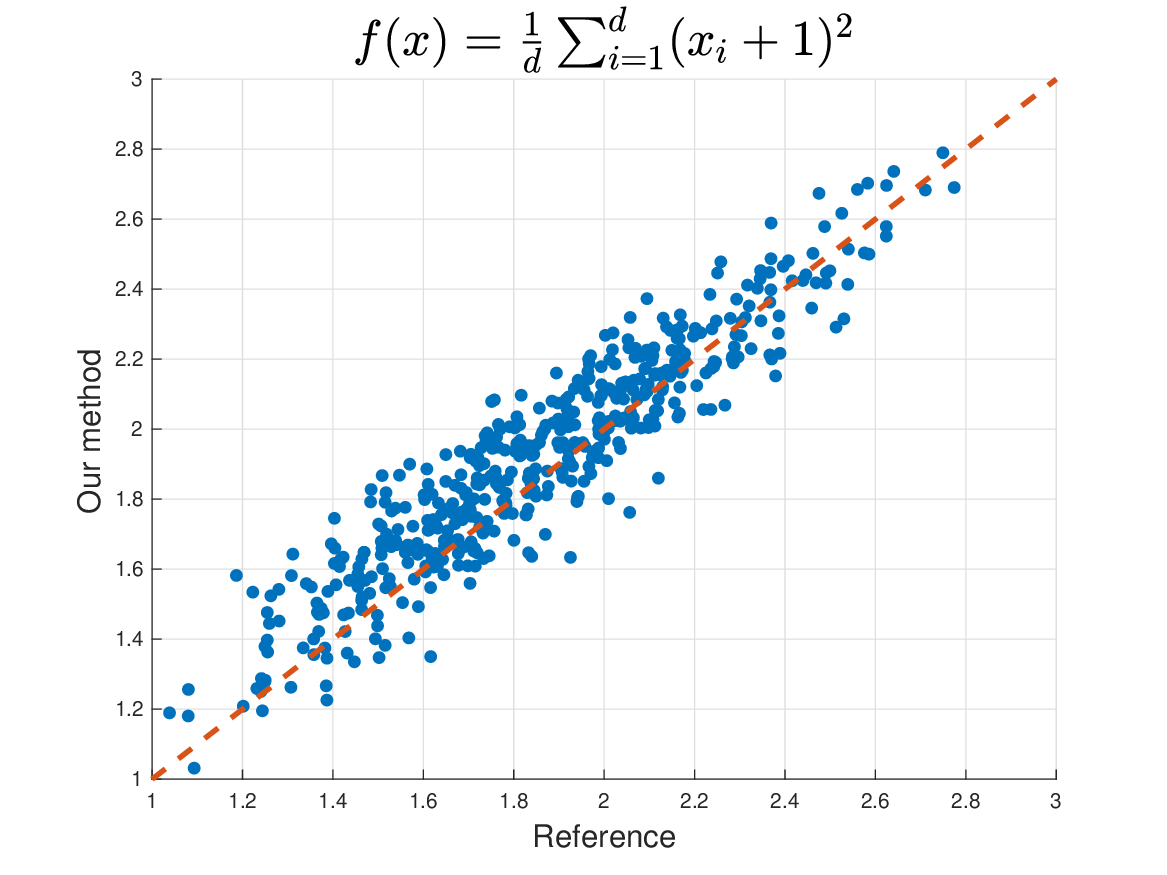}
    \caption{Comparison between function evaluations obtained by our method and the Monte Carlo reference for the moment prediction described in Section~\ref{sec:backward_results}.}
\label{fig:backward}
\end{figure}

We compare the solutions produced by our method with direct Monte Carlo simulations.  Specifically, we randomly sample 500 testing points from the hypercube domain $\Omega = [-L,L]^d$ with \(L = 2\), and evaluate the computed predictions at these points using our algorithm. For the reference solution, the evaluation at each testing point is obtained by averaging over 100 trajectories simulated under the overdamped Langevin dynamics. 
The results are presented in \Cref{fig:backward}, where our computed solutions and Monte Carlo references are plotted against each other. For both \(f_+\) and \(f_-\), our solution agree well with the Monte Carlo evaluations.
The relative \(l^2\) errors for \(f_+\) and \(f_-\) are \(0.084\) and \(0.070\), respectively.

\subsection{Density prediction}
\label{sec:forward_results}
We apply the proposed method to forward prediction, as discussed in \Cref{sec:forward}.  Specifically, we set the initial density to be the mean-field Boltzmann distribution
\begin{displaymath}
    \rho_0(x)  = \frac{1}{Z_\beta} \exp\left( - \beta  \sum_{i=1}^{d}\frac{1}{4\gamma} (1-x_i^2)^2 \right)
\end{displaymath}
where $Z_\beta$ is the normalization constant.
We can use the compressed transition moment matrix with $\mu=\rho_0$ to predict the evolved density \(\rho_T(x)\) under the Langevin dynamics associated with \(V(x)\) for a total simulation time \(T\).

We consider an example with $d=50$, coupling intensity $\gamma = 0.005$, inverse temperature $\beta = 1/20$ and predict the density at $T=0.1$. In the algorithm, we set the number of source samples $N_{\text{src}} = 5000$, number of trajectories for each source sample $N_{\text{traj}} = 100$, and use first 8 Legendre bases as the univariate bases for the cluster bases. 
In Step 1 of \Cref{alg:two_level}, we set the CUR rank $r_1=500$. In Step 2 of \Cref{alg:two_level}, we choose bandwidth $\delta=10$ for the sparse factor, and implement truncated SVD to compress the low-rank factor, with rank $r_2$ determined adaptively by the SVD truncation error $ 10^{-3}$.

\begin{figure}[ht!]
\centering
\includegraphics[width=0.5\textwidth]{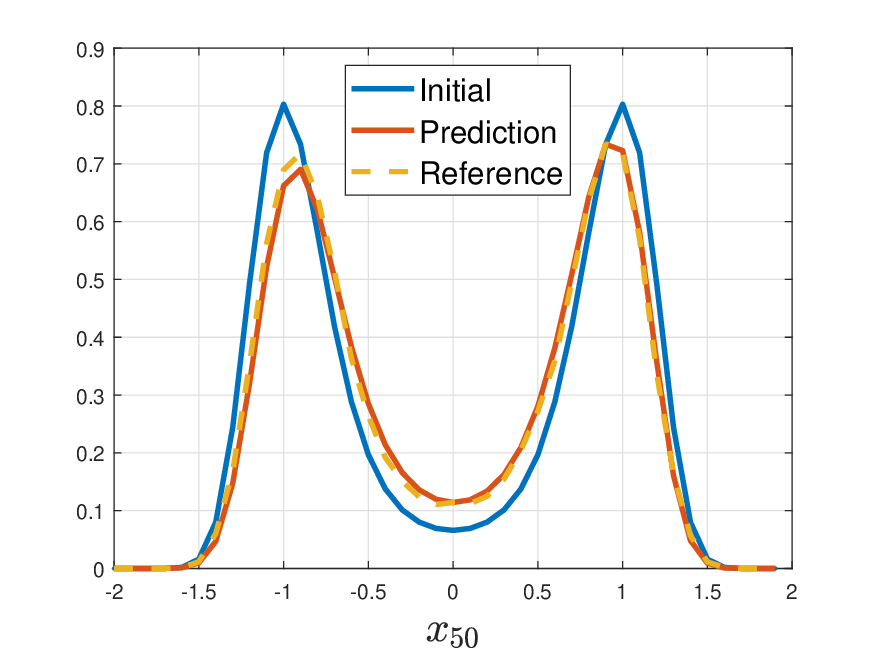}
    \caption{$x_{50}$-marginal plots of initial density, predicted density and KDE reference, obtained in Section~\ref{sec:forward_results}.}
\label{fig:forward_mar1}
\end{figure}

\begin{figure}[ht!]
\centering
\includegraphics[width=\textwidth]{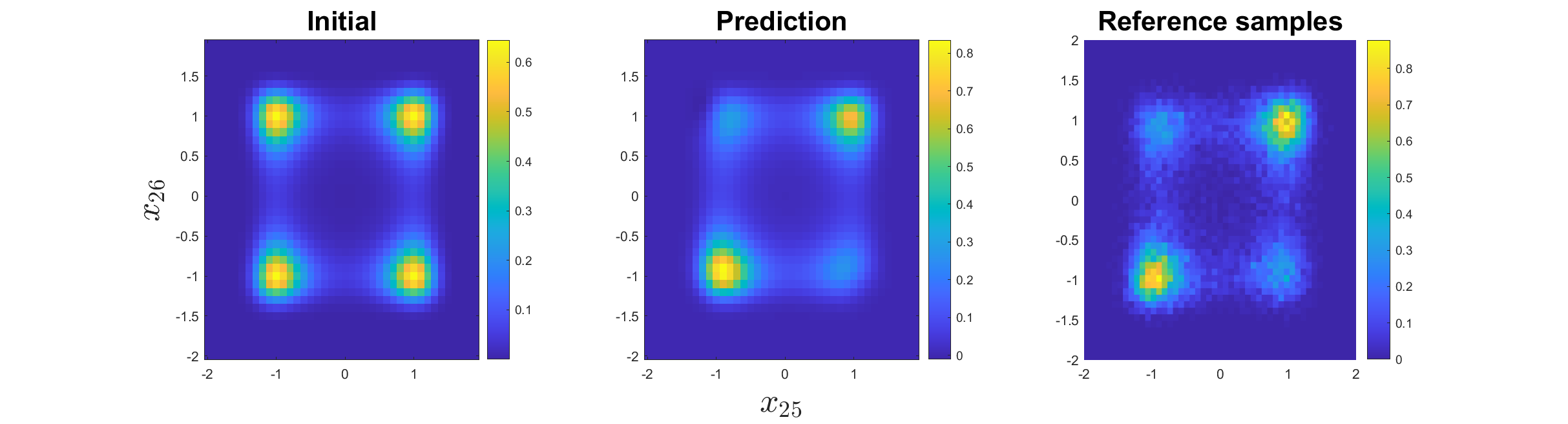}
\caption{$(x_{25},x_{26})$-marginal plots of initial density, predicted density and reference sample histograms, obtained in Section~\ref{sec:forward_results}.}
\label{fig:forward_mar2}
\end{figure}

In \Cref{fig:forward_mar1}, we visualize the $x_{50}$-marginal plot of the predicted density. For comparison, we draw 20,000 samples directly from the overdamped Langevin dynamics and perform kernel density estimates (KDE) on the sample marginal distribution. As shown in the figure, our prediction exhibits decent agreement with the reference KDE curve. The relative \(l^2\) errors between the prediction and KDE solution is $0.0625$. To further assess the spatial structure of the predicted density, we visualize the two-dimensional marginals in \Cref{fig:forward_mar2}. One can observe that the prediction well captures the locations of the local minima and reproduces the overall shape of the reference distribution, demonstrating that our method can recover high-dimensional metastable structures with high fidelity.

\subsection{Committor function}
\label{sec:committor_results}
Finally, we present our numerical results for solving committor function defined by the boundary value problem \eqref{eqn:committor}, where we let $A$ and $B$ respectively be the ball centered at $x_+$ and $x_-$ with radius $R$: $\{x:\|x-x_{\pm}\|_F \leqslant R\}$.

In this experiment, we set model parameters $d = 50$, coupling intensity $\gamma=0.1$ and inverse temperature $\beta=1/16$ and ball radius $R=2.5$. For our algorithm, we use lag time $t = 0.2$, number of source samples $N_{\text{src}} = 5000$, number of trajectories for each source samples $N_{\text{traj}} = 200$, number of samples on boundaries $N_{bc} = 1000$, and use first 8 Fourier bases as the univariate bases for the cluster bases. In Step 1 of \Cref{alg:two_level}, we set the CUR rank $r_1=600$. In Step 2 of \Cref{alg:two_level}, we choose bandwidth $\delta=5$ for the sparse factor, and implement truncated SVD to compress the low-rank factor, with rank $r_2$ determined adaptively by the SVD truncation error $ 10^{-3}$.  

We validate our algorithm by two tests in Figure~\ref{fig:committor_validation_weak}. In the first test, we examine the one-dimensional slice $q(x,x,\dots,x)$ along the diagonal path connecting $(1,1,\cdots,1)$ and $(-1,-1,\cdots,-1)$, which approximates the two global minimum $x_+$ and $x_-$ of potential energy \eqref{gl1d_potential}. As illustrated in the left panel of Figure~\ref{fig:committor_validation_weak}, our solution demonstrates a transition in the middle of the slice, indicating that a trajectory has high likelihood to enter $x_\pm$ when $x$ is away from zero. To assess the accuracy of the proposed method, we compare the estimated solutions of \( q(x,x,\dots,x) \) with reference solution solved by direct Monte Carlo simulations, using $10{,}000$ samples  for each starting point \((x,x,\dots,x)\). Our method achieves a relative error of $0.002$. As another check illustrated in the right panel of Figure~\ref{fig:committor_validation_weak}, we randomly sample 500 points from $\rho_\infty$ and evaluate the committor function at these points using our method.  For the Monte Carlo reference, we use 100 trajectories at each point. The resulting relative error is $0.049$.

 \begin{figure}[ht!]
\centering
\includegraphics[width=0.49\textwidth]{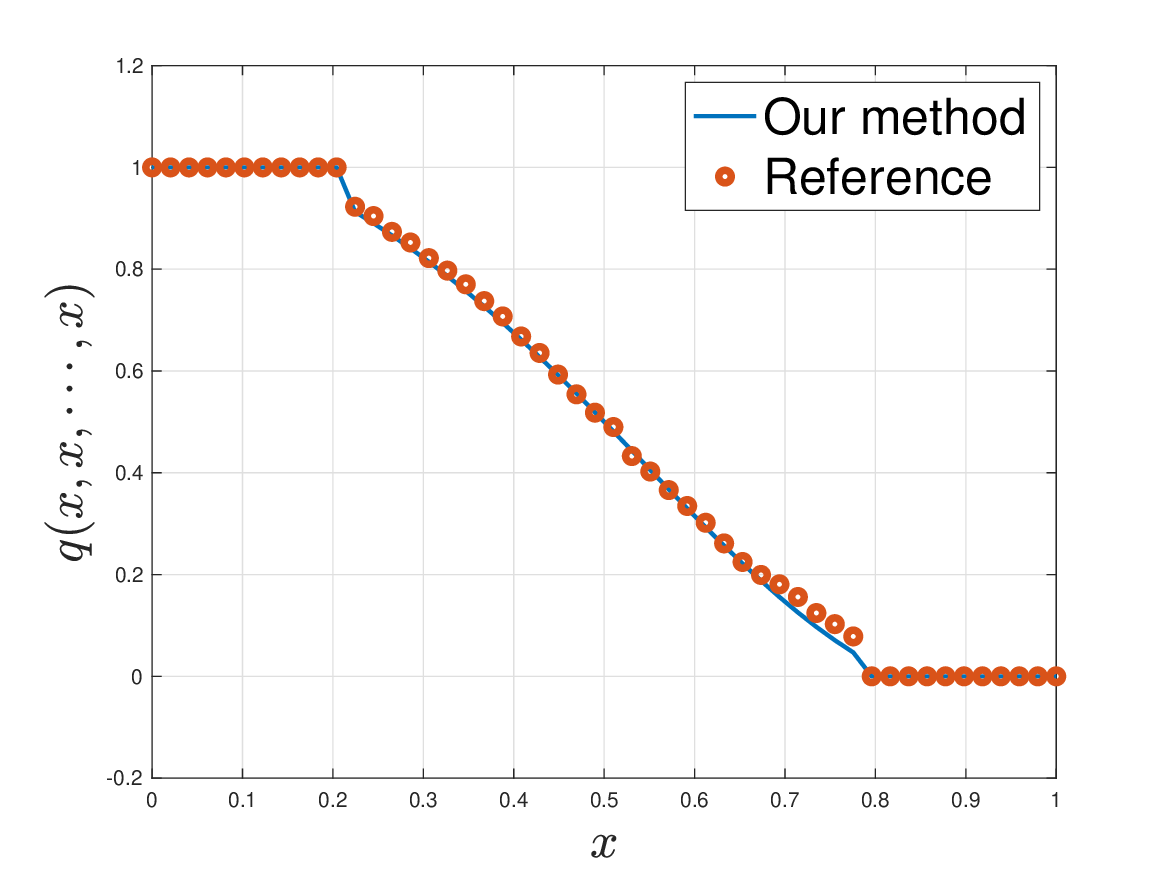}
\includegraphics[width=0.49\textwidth]{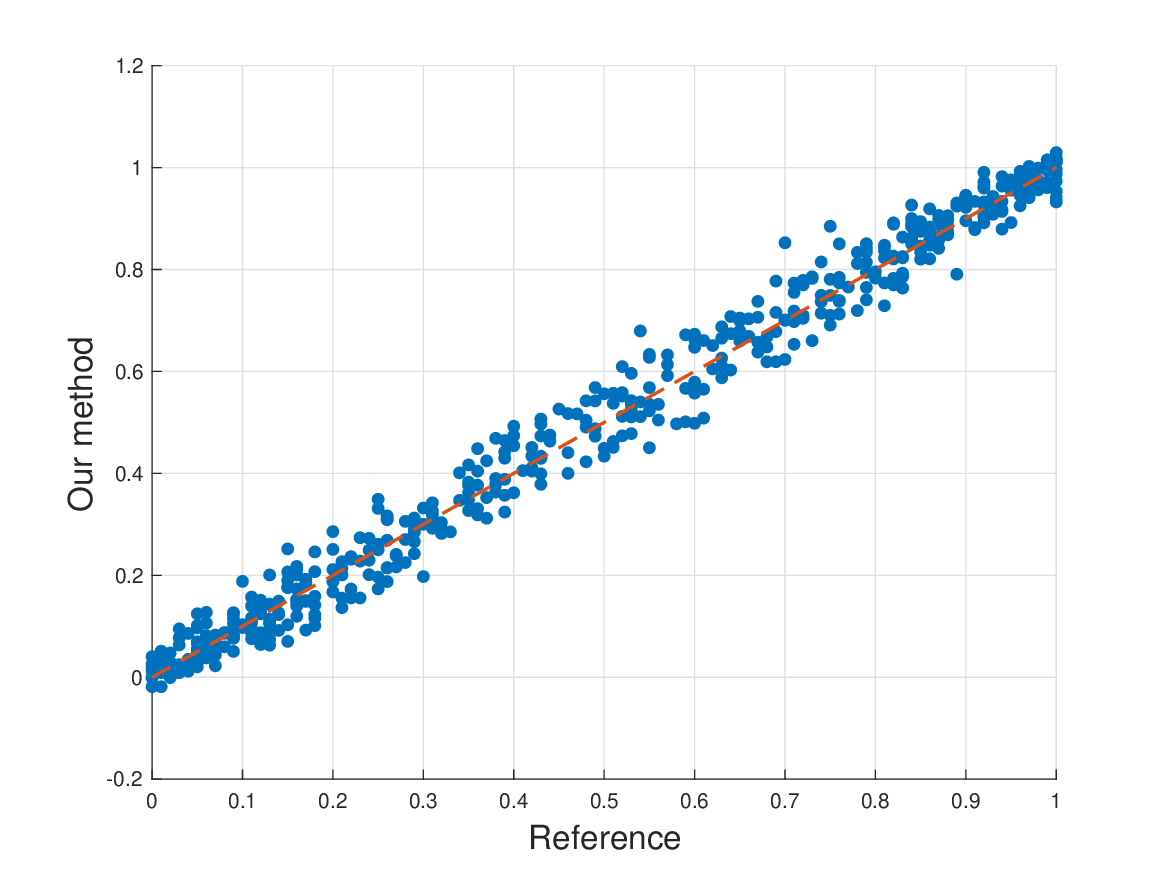}
    \caption{Committor function problem in Section~\ref{sec:committor_results}. Left: comparison of the committor function $q(x,x,\ldots,x)$  using our method and the Monte Carlo reference (with $x$ rescaled to $[0,1]$).  Right: pointwise comparison between our method and the Monte Carlo reference at  sampled points (with $x$ rescaled to $[0,1]$).}
\label{fig:committor_validation_weak}
\end{figure}

\section{Conclusion}
\label{sec:diccusion}
In this work, we show that the transition moment matrix~\eqref{transition matrix} provides an effective finite-dimensional surrogate for the transition operator associated with high-dimensional stochastic dynamics. By exploiting both the global low-rank structure of the transition operator at moderate lag times and the local compressed representation induced by the decay-of-correlation property, we develop a two-level compression scheme that enables efficient construction and storage of the transition moment matrix. Our theoretical analysis, based on a reversible finite-state Markov model,  establishes exponential decay of correlations between distant coordinates, and rigorously proves the existence of the two-level compressed representation. Our numerical results demonstrate that the learned transition moment matrix can be used to solve momenr prediction, density prediction and committor functions in high dimensions with decent accuracies, without relying on any optimization.

We highlight several natural directions for future work. First, it would be interesting to integrate the proposed framework with appropriately chosen collective variables (CV) and apply it to more complex molecular dynamics systems for transition-path analysis. In addition, although the current representation already permits efficient matrix–vector multiplications, incorporating tensor structures into the coefficient representation may further reduce computational cost and improve the efficiency of iterative solvers. On the theoretical side, our current analysis does not account for Monte Carlo sampling error. Furthermore, the present work derives decay-of-correlation property based on finite-state models under the equilibrium measure for simplicity. Extending these theoretical guarantees to continuous-state models with more general measures while rigorously incorporating Monte Carlo error would represent a significant advancement.

\section*{Acknowledgment}
Y. Khoo was partially funded by NSF DMS-2339439, DOE DE-SC0022232, DARPA The Right Space HR0011-25-9-0031, and a Sloan research fellowship.

\appendix

\section{An example of two-cluster basis functions}
\label{app:idx}

\begin{table}[ht!]
\centering
\scriptsize
\setlength{\tabcolsep}{4pt}
\renewcommand{\arraystretch}{1.4}

\begin{tabular}{c|cccc|cccc|cccc}
\hline
$j$
 & 1 & 2 & 3 & 4
 & 5 & 6 & 7 & 8
 & 9 & 10 & 11 & 12 \\
\hline
$s(j)$
 & (1,1)&(1,1)&(1,1)&(1,1)
 & (1,2)&(1,2)&(1,2)&(1,2)
 & (1,3)&(1,3)&(1,3)&(1,3) \\
\hline
$b(j)$
 & (1,1)&(1,2)&(2,1)&(2,2)
 & (1,1)&(1,2)&(2,1)&(2,2)
 & (1,1)&(1,2)&(2,1)&(2,2) \\
\hline
$\psi_j(x)$
 & $\phi^1_1\phi^1_1$ & $\phi^1_1\phi^1_2$ & $\phi^1_2\phi^1_1$ & $\phi^1_2\phi^1_2$
 & $\phi^1_1\phi^2_1$ & $\phi^1_1\phi^2_2$ & $\phi^1_2\phi^2_1$ & $\phi^1_2\phi^2_2$
 & $\phi^1_1\phi^3_1$ & $\phi^1_1\phi^3_2$ & $\phi^1_2\phi^3_1$ & $\phi^1_2\phi^3_2$ \\
\hline
\end{tabular}

\vspace{1em}

\begin{tabular}{c|cccc|cccc|cccc}
\hline
$j$
 & 13 & 14 & 15 & 16
 & 17 & 18 & 19 & 20
 & 21 & 22 & 23 & 24 \\
\hline
$s(j)$
 & (2,1)&(2,1)&(2,1)&(2,1)
 & (2,2)&(2,2)&(2,2)&(2,2)
 & (2,3)&(2,3)&(2,3)&(2,3) \\
\hline
$b(j)$
 & (1,1)&(1,2)&(2,1)&(2,2)
 & (1,1)&(1,2)&(2,1)&(2,2)
 & (1,1)&(1,2)&(2,1)&(2,2) \\
\hline
$\psi_j(x)$
 & $\phi^2_1\phi^1_1$ & $\phi^2_1\phi^1_2$ & $\phi^2_2\phi^1_1$ & $\phi^2_2\phi^1_2$
 & $\phi^2_1\phi^2_1$ & $\phi^2_1\phi^2_2$ & $\phi^2_2\phi^2_1$ & $\phi^2_2\phi^2_2$
 & $\phi^2_1\phi^3_1$ & $\phi^2_1\phi^3_2$ & $\phi^2_2\phi^3_1$ & $\phi^2_2\phi^3_2$ \\
\hline
\end{tabular}

\vspace{1em}

\begin{tabular}{c|cccc|cccc|cccc}
\hline
$j$
 & 25 & 26 & 27 & 28
 & 29 & 30 & 31 & 32
 & 33 & 34 & 35 & 36 \\
\hline
$s(j)$
 & (3,1)&(3,1)&(3,1)&(3,1)
 & (3,2)&(3,2)&(3,2)&(3,2)
 & (3,3)&(3,3)&(3,3)&(3,3) \\
\hline
$b(j)$
 & (1,1)&(1,2)&(2,1)&(2,2)
 & (1,1)&(1,2)&(2,1)&(2,2)
 & (1,1)&(1,2)&(2,1)&(2,2) \\
\hline
$\psi_j(x)$
 & $\phi^3_1\phi^1_1$ & $\phi^3_1\phi^1_2$ & $\phi^3_2\phi^1_1$ & $\phi^3_2\phi^1_2$
 & $\phi^3_1\phi^2_1$ & $\phi^3_1\phi^2_2$ & $\phi^3_2\phi^2_1$ & $\phi^3_2\phi^2_2$
 & $\phi^3_1\phi^3_1$ & $\phi^3_1\phi^3_2$ & $\phi^3_2\phi^3_1$ & $\phi^3_2\phi^3_2$ \\
\hline
\end{tabular}
\caption{Two-cluster basis functions $\{\psi_j\}_{j\in [(dn)^2]}$ defined in \eqref{two cluster linear} for $d=3$, $n=2$. Each $\phi_b^s$ is an univariate basis acting on $x_s$ and $b$ indicates the order of basis.}
\end{table}


\section{Detailed Proof of \Cref{lemma:decay of correlation}}
\label{app:proof}

\subsection{Supportive results}
Before proving \Cref{lemma:decay of correlation}, we introduce several supportive results which will be useful for the proof. We start from introducing {\it Dirichlet form}, which  plays a central role in our analysis. 
\begin{definition}\cite[Eqn 1.7.1]{bakry2014analysis}
    Given functions $f,g$ defined on $\Omega$, the Dirichlet form and local Dirichlet form under equilibrium density are respectively defined as 
    \begin{equation}\label{dirichlet form}
     \cE(f,g) = \sum_{i=1}^d \cE_i(f,g) =  - \langle \cL f, g\rangle_{\rho_\infty}, \quad        \cE_i(f,g) = - \langle \cL_i f, g\rangle_{\rho_\infty} . 
    \end{equation}
\end{definition}
Since $\cL1=0$ by \Cref{assum:spec gap}, it is easy to verify that the Dirichlet form is invariant to constant shifts: 
\begin{align}
\label{eqn:invariance_constant}
    \cE(f+c_1,g+c_2)=\cE(f,g), \qquad \forall c_1,c_2\in\R. 
\end{align}
For reversible Markov processes, the Dirichlet form admits an equivalent symmetric representation \cite[Theorem 4.2]{funaki1869stochastic}, as stated below. 

\begin{proposition}
\label{lma:symmetric_rep}
If the Markov process is reversible, then the local Dirichlet form has the following symmetric representation 
\begin{align}
\label{eqn:symmetric_local_dirichlet}
    \cE_i(f,g)
    = \frac{1}{2}\sum_{x\in \cX^d}\sum_{\alpha\in \cX}r_{i,\alpha}(x)
      \Delta_{i,\alpha} f(x)\,\Delta_{i,\alpha} g(x)\,\rho_\infty(x).
\end{align}
\end{proposition}
\begin{proof}
By definition,
\begin{displaymath}
    \cE_i(f,g)
    = -\sum_{x\in \cX^d}\sum_{\alpha\in \cX}
      r_{i,\alpha}(x)\,\Delta_{i,\alpha}f(x)\,g(x)\,\rho_\infty(x).
\end{displaymath}
Let
\begin{equation}\label{S}
 \begin{split}
    T
    &= \sum_{x\in \cX^d}\sum_{\alpha\in \cX}
       r_{i,\alpha}(x)\,\Delta_{i,\alpha}f(x)\,\Delta_{i,\alpha}g(x)\,\rho_\infty(x) \\[0.3em]
    &= \sum_{x\in \cX^d}\sum_{\alpha\in \cX}
       r_{i,\alpha}(x)\Big(
       f(\tau_{i,\alpha}x)g(\tau_{i,\alpha}x)
       - f(\tau_{i,\alpha}x)g(x)
       - f(x)g(\tau_{i,\alpha}x)
       + f(x)g(x)
       \Big)\rho_\infty(x).
 \end{split}
\end{equation}

By detailed balance condition~\eqref{eq:detailed_balance}, the third term in \eqref{S} can be transformed into the second as
\begin{align*}
   \sum_{x\in \cX^d}\sum_{\alpha\in \cX}
   r_{i,\alpha}(x)\,f(x)g(\tau_{i,\alpha}x)\,\rho_\infty(x)
   &= \sum_{x\in \cX^d}\sum_{\alpha\in \cX}
      r_{i,x_i}(\tau_{i,\alpha}x)\,f(x)g(\tau_{i,\alpha}x)\,\rho_\infty(\tau_{i,\alpha}x)\\
   &= \sum_{x_i\in \cX}\sum_{x_{-i}\in \cX^{d-1}}\sum_{\alpha\in \cX}
      r_{i,x_i}(\tau_{i,\alpha}x)\,f(x)g(\tau_{i,\alpha}x)\,\rho_\infty(\tau_{i,\alpha}x).
\end{align*}
By interchanging the dummy variables \(x_i\) and \(\alpha\), we obtain
\begin{align}
\label{S23}
\begin{split}
 \sum_{x\in \cX^d}\sum_{\alpha\in \cX}
 r_{i,\alpha}(x)\,f(x)g(\tau_{i,\alpha}x)\,\rho_\infty(x)  &= 
 \sum_{\alpha\in \cX}\sum_{x_{-i}\in \cX^{d-1}}\sum_{x_i\in \cX}
      r_{i,\alpha}(x)\,f(\tau_{i,\alpha}x)g(x)\,\rho_\infty(x)
 \\
  &= \sum_{x\in \cX^d}\sum_{\alpha\in \cX}
   r_{i,\alpha}(x)\,f(\tau_{i,\alpha}x)g(x)\,\rho_\infty(x),
\end{split}   
\end{align}
showing that the second and third terms in \eqref{S} are identical. Similarly, applying the same argument to the first term gives
\begin{align}
\label{S14}
 \sum_{x\in \cX^d}\sum_{\alpha\in \cX}
 r_{i,\alpha}(x)\,f(\tau_{i,\alpha}x)g(\tau_{i,\alpha}x)\,\rho_\infty(x)
 = \sum_{x\in \cX^d}\sum_{\alpha\in \cX}
   r_{i,\alpha}(x)\,f(x)g(x)\,\rho_\infty(x),
\end{align}
so that the first and last terms in \eqref{S} also coincide. Inserting \eqref{S23} and \eqref{S14} into \eqref{S} yields \eqref{eqn:symmetric_local_dirichlet} since 
\begin{align*}
    T
    &= 2\sum_{x\in \cX^d}\sum_{\alpha\in \cX}
       r_{i,\alpha}(x)\,[f(x)-f(\tau_{i,\alpha}x)]\,g(x)\,\rho_\infty(x)\\
    &= -2\sum_{x\in \cX^d}\sum_{\alpha\in \cX}
       r_{i,\alpha}(x)\,\Delta_{i,\alpha}f(x)\,g(x)\,\rho_\infty(x)
    = 2\,\cE_i(f,g).
\end{align*}
\end{proof}
Based on the symmetric representation~\eqref{eqn:symmetric_local_dirichlet}, we immediately obtain the following properties.

\begin{corollary}
    If the Markov process is reversible, then the Dirichlet form and the local Dirichlet form are symmetric i.e. 
    \begin{align}
        \cE(f,g)=\cE(g,f), \qquad \cE_i(f,g)=\cE_i(g,f)
    \end{align}
    and non-negative definite, i.e. 
    \begin{align}
        \cE(f,f)\geqslant 0, \qquad \cE_i(f,f)\geqslant 0
    \end{align}
    for any function $f,g$ defined on $\Omega$. 
\end{corollary}

We can also establish the Cauchy–Schwarz inequality for the local Dirichlet form.

\begin{proposition}[Cauchy-Schwarz inequality]
\label{prop:cauchy_Schwarz}
    For any functions $f$, $g$ defined on $\Omega$, 
    \begin{align}
    \label{eqn:cauchy schwarz}
        |\cE_i(f,g)|\leqslant \sqrt{\cE_i(f,f) \cE_i(g,g)}.
    \end{align}
\end{proposition}

\begin{proof}
By the symmetric representation~\eqref{eqn:symmetric_local_dirichlet}, $\forall t\in \R$, 
\begin{align*}
   0\leqslant  \cE_i(f-tg,f-tg)
&=\frac{1}{2}\sum_{x\in \cX^d}\sum_{\alpha\in \cX}
r_{i,\alpha}(x)\,(\Delta_{i,\alpha}(f(x)-tg(x))^2\,\rho_\infty(x) \\ 
&=\frac{1}{2}  \sum_{x\in \cX^d}\sum_{\alpha\in \cX} 
r_{i,\alpha}(x)( \Delta_{i,\alpha} f(x))^2 \rho_\infty(x)\\
&+\sum_{x\in \cX^d}\sum_{\alpha\in \cX} 
r_{i,\alpha}(x)  \Delta_{i,\alpha} f(x)  \Delta_{i,\alpha} g(x) \rho_\infty(x)\\ 
& + \frac{t^2}{2} \sum_{x\in \cX^d}\sum_{\alpha\in \cX} 
r_{i,\alpha}(x) ( \Delta_{i,\alpha} g(x))^2 \rho_\infty(x)\\
&=\frac{1}{2} \cE_i(f,f)+ \cE_i(f,g) t+\frac{1}{2}  \cE_i(g,g)t^2 := h(t).
\end{align*}
The discriminant of $h(t)$ must be non-positive, which leads to the Cauchy-Schwarz inequality~\eqref{eqn:cauchy schwarz}. 
\end{proof}

It is easy to see that $|\Delta_{i,\alpha}f|\leqslant \Osc_i(f)$ for any function $f$ defined on $\Omega$, which yields the following upper bound on the local Dirichlet form. 

\begin{proposition}[Local energy bound]
\label{prop:energy_bound}
  For any functions $f,g$ defined on $\Omega$, we have 
  \begin{align}
  \label{eqn:local_energy_bound}
      |\cE_i(f,g)|\leqslant \frac{1}{2} |\cX| r_{\max} \Osc_i (f) \Osc_i(g).
  \end{align}
\end{proposition}

\begin{proof}
By the symmetric representation~\eqref{eqn:symmetric_local_dirichlet},
\begin{align*}
   | \cE_i(f,g)|
    & \leqslant  \frac{1}{2}\sum_{x\in \cX^d}\sum_{\alpha\in \cX}r_{i,\alpha}(x)
      |\Delta_{i,\alpha} f(x)|\,|\Delta_{i,\alpha} g(x)|\,\rho_\infty(x) \\ 
     &  \leqslant \frac{1}{2}\sum_{x\in \cX^d}\sum_{\alpha\in \cX}r_{i,\alpha}(x)
    \Osc_i(f)\Osc_i(g)\,\rho_\infty(x)\\
      &\leqslant \frac{1}{2}|\cX| r_{\max}\Osc_i(f)\Osc_i(g) \sum_{x\in \cX^d} \rho_\infty(x).
\end{align*}
The result follows since $\sum_{x\in \cX^d} \rho_\infty(x)=1.$
\end{proof}

Besides the local energy bound, we also have the following Poincar\'e inequality bounding the Dirichlet energy of $\cP_t f$:
\begin{proposition}[Poincar\'e inequality]
    Under \Cref{assum:spec gap}, for any function $f$ defined on $\Omega$, it holds that 
    \begin{equation}
    \label{eqn:poincare}
        \cE(\cP_t f, \cP_t f) \leqslant e^{-2\lambda t}  \cE( f,  f)
    \end{equation}
    for any $t\geqslant 0$, where $\lambda$ is the smallest non-zero eigenvalue of $-\cL$.
\end{proposition}
\begin{proof}
Let \(\psi_t := \cP_t f = e^{t\cL} f\) and define its mean-zero part as  $$\tilde \psi_t=\psi_t-\langle \psi_t \rangle_{\rho_\infty}.$$ 
Then $\langle \tilde \psi_t\rangle_{\rho_\infty}=0$. Since $\cL 1=0$ according to \Cref{assum:spec gap}, we have $\cL\psi_t=\cL \tilde \psi_t$.
Moreover, since $\rho_\infty$ is the density of equilibrium distribution, $$\langle \psi_t \rangle_{\rho_\infty}=\langle \cP_t f  , \rho_\infty\rangle =\langle f, \cP_t^* \rho_\infty \rangle=\langle f, \rho_\infty \rangle=\langle f\rangle_{\rho_\infty}$$
so the mean is constant in $t$.
Hence $\tilde \psi_t$ satisfies
\begin{align*}
    \partial_t \tilde \psi_t = \cL \tilde \psi_t, \qquad 
    \tilde \psi_0=f-\langle f \rangle_{\rho_\infty}.
\end{align*}
Define $E(t)=\cE(\psi_t,\psi_t)=\cE(\tilde \psi_t,\tilde \psi_t)=-\langle \cL \tilde \psi_t, \tilde \psi_t \rangle_{\rho_\infty}$, where the second equality uses the invariance \eqref{eqn:invariance_constant} of $\cE$.  Differentiating $E(t)$ and using the self-adjointness of $\cL$ gives
\begin{align*}
    E'(t)  =-\langle \cL \partial_t \tilde \psi_t, \tilde \psi_t \rangle_{\rho_\infty}-\langle \cL \tilde \psi_t, \partial_t \tilde \psi_t \rangle_{\rho_\infty}  
    = - \langle \cL^2 \tilde \psi_t, \tilde \psi_t \rangle_{\rho_\infty}
    -\langle \cL \tilde \psi_t, \cL \tilde \psi_t \rangle_{\rho_\infty}  
    =-2\langle \cL^2 \tilde \psi_t, \tilde \psi_t \rangle_{\rho_\infty}. 
\end{align*}
Now, by~\Cref{assum:spec gap}, for all mean-zero function $u$,
\begin{align*}
    \langle (-\cL)^2 u,u\rangle_{\rho_\infty}\geqslant \lambda \langle (-\cL) u,u \rangle_{\rho_\infty}=\lambda \cE(u,u).
\end{align*}
Applying $u=\tilde \psi_t$ gives $E'(t)\leqslant -2\lambda E(t)$. Hence,
$$
E(t)\leqslant e^{-2\lambda t}E(0)=e^{-2\lambda t} \cE(f-\langle f\rangle_{\rho_\infty},f-\langle f\rangle_{\rho_\infty})=e^{-2\lambda t} \cE(f,f)
$$
where the last equality also follows from the invariance of $\cE$ under adding constants.
\end{proof}

Finally, the following bound for Dini derivative of oscillation of $\cP_t f$ will also be useful in our analysis. 
\begin{proposition}\label{lemma:Dini}
For function $f$ defined on $\Omega$, 
    \begin{equation}
        D^+ \Osc_i(\cP_t f) :=   \limsup_{\epsilon \rightarrow 0^+} \frac{ \Osc_i(\cP_{t+\epsilon} f) -  \Osc_i(\cP_t f)}{\epsilon}\leqslant \Osc_i(\cL \cP_t f)
    \end{equation}
\end{proposition}
\begin{proof}
For any $\epsilon >0$, 
\begin{equation*}
    \begin{split}
   & \Osc_i(\cP_{t+\epsilon}f)  - \Osc_i(\cP_{t}f) \\ 
   = \ & \sup_{x_{[d]\backslash i}} \sup_{x'_i,x''_i} \left| \cP_{t+\epsilon}f(x'_i,x_{[d]\backslash i}) - \cP_{t+\epsilon}f(x''_i,x_{[d]\backslash i}) \right| - \sup_{x_{[d]\backslash i}} \sup_{x'_i,x''_i} \left| \cP_{t}f(x'_i,x_{[d]\backslash i}) - \cP_{t}f(x''_i,x_{[d]\backslash i}) \right| \\
   \leqslant \ & \sup_{x_{[d]\backslash i}} \sup_{x'_i,x''_i} \left| \cP_{t+\epsilon}f(x'_i,x_{[d]\backslash i}) - \cP_{t+\epsilon}f(x''_i,x_{[d]\backslash i}) \right| -  \left| \cP_{t}f(x'_i,x_{[d]\backslash i}) - \cP_{t}f(x''_i,x_{[d]\backslash i}) \right| \\
   \leqslant \ &  \sup_{x_{[d]\backslash i}} \sup_{x'_i,x''_i} \left| (\cP_{t+\epsilon}f(x'_i,x_{[d]\backslash i}) - \cP_{t}f(x'_i,x_{[d]\backslash i}) )-(\cP_{t+\epsilon}f(x''_i,x_{[d]\backslash i})   - \cP_{t}f(x''_i,x_{[d]\backslash i}) ) \right|.
    \end{split}
\end{equation*}
Therefore, 
\begin{align*}
     D^+ \Osc_i(\cP_t f) 
     \leqslant   \sup_{x_{[d]\backslash i}} \sup_{x'_i,x''_i} \left| \cL\cP_t f(x'_i,x_{[d]\backslash i}) -  \cL\cP_t f(x''_i,x_{[d]\backslash i}) \right| = \Osc_i(\cL \cP_tf). 
\end{align*}
The last inequality uses the fact that $\cL\cP_t = \lim_{\epsilon\to 0+} \frac{\cP_{t+\epsilon}-\cP_t}{\epsilon}$.
\end{proof}

With the above preparations, we are now ready to prove \Cref{lemma:decay of correlation} in the next subsection. 

\subsection{Final Proof of \Cref{lemma:decay of correlation}}
\label{sec:proof_of_decay_of_correlation}
We begin by establishing the following Lieb-Robinson bound. 
\begin{lemma}[Lieb-Robinson bound] Under \Cref{assum:finite range}, for function $f$ defined on $\Omega$, it holds that 
 \label{prop:LR bound}
   \begin{equation}
   \label{eqn:LR_bound}
           \Osc_i(\cP_t f) \leqslant  \sum_{j=1}^d \exp\left(-\frac{|i-j| - v t}{\kappa}\right)  \Osc_j(f) 
   \end{equation}
where $v = 2(e - 1)(2\kappa+1)\kappa ( |\cX| r_{\max} + \bar{r}) $.
\end{lemma}

\begin{proof}
\noindent 
\textbf{Step 1}: Upper bound of $\Osc_i(\cL f)$. 

By definition of the oscillation \eqref{def:osc}, we need to estimate $(\cL f)(x^{(i)}) - (\cL f)(\tilde x^{(i)})$, where the configurations $x^{(i)}$ and $\tilde x^{(i)}$ differ only at the $i$-th site (coordinate). Inserting definition of the generator \eqref{def:generator2}, we have 
\begin{equation}\label{osc_expansion}
    \begin{split}
        & (\cL f)(x^{(i)}) - (\cL f)(\tilde x^{(i)}) \\
         = \ &  \sum_{j=1}^d \sum_{\alpha \in \cX}    r_{j,\alpha}(x^{(i)}) \Delta_{j,\alpha} f(x^{(i)})  -  r_{j,\alpha}(\tilde x^{(i)}) \Delta_{j,\alpha} f(\tilde x^{(i)}) \\
        = \ &    \underbrace{\sum_{j=1}^d \sum_{\alpha\in \cX} r_{j,\alpha}(x^{(i)}) \left( \Delta_{j,\alpha} f(x^{(i)})  -  \Delta_{j,\alpha} f(\tilde x^{(i)}) \right)}_{=: T_1} +  \underbrace{ \sum_{j=1}^d \sum_{\alpha\in \cX}\Delta_{j,\alpha} f(\tilde x^{(i)}) \left(  r_{j,\alpha}(x^{(i)}) -  r_{j,\alpha}(\tilde x^{(i)})  \right)}_{=:T_2}.
    \end{split}
\end{equation}
To estimate $T_1$, first apply definition \eqref{def:Delta} of $\Delta_{j,\alpha}$ to get 
\begin{equation*}
 \begin{split}
  &  r_{j,\alpha}(x^{(i)}) \left( \Delta_{j,\alpha} f(x^{(i)})  -  \Delta_{j,\alpha} f(\tilde x^{(i)}) \right) \\
  \leqslant   \ &   r_{j,\alpha}(x^{(i)}) \left(  \left|f(\tau_{j,\alpha}x^{(i)}) - f(\tau_{j,\alpha} \tilde x^{(i)})    \right|  +  \left| f(x^{(i)}) - f(\tilde x^{(i)})    \right| \right) 
 \leqslant  2   r_{j,\alpha}(x^{(i)}) \Osc_i (f).
     \end{split}
\end{equation*}
Hence,  
\begin{equation}\label{T1 bound}
    |T_1| \leqslant 2 \bar{r}  \Osc_i (f)
\end{equation}
where $\bar r$ is defined in \eqref{rmax and rbar}. 

For $T_2$, we use \Cref{assum:finite range} to bound 
\begin{equation}\label{T2 bound}
 \begin{split}
    |T_2| \leqslant \ &  \sum_{j=1}^d \sum_{\alpha\in \cX} |\Delta_{j,\alpha} f(\tilde x^{(i)}) |  \cdot |r_{j,\alpha}(x^{(i)}) -  r_{j,\alpha}(\tilde x^{(i)})| \cdot \mathbf{1}_{\{|i-j|\leqslant \kappa\}} \\
     \leqslant \ &   \sum_{j=1}^d \sum_{\alpha\in \cX} \Osc_j(f)  \cdot 2 r_{\max} \cdot \mathbf{1}_{\{|i-j|\leqslant \kappa\}} =  2|\cX| r_{\max}   \sum_{j=1}^d  \Osc_j(f)  \cdot \mathbf{1}_{\{|i-j|\leqslant \kappa\}}.
     \end{split}
\end{equation}

Combine \eqref{osc_expansion}, \eqref{T1 bound} with \eqref{T2 bound}, we obtain the following upper bound for $\Osc_i(\cL f)$:
\begin{equation}\label{LR_bd_step_1}
    \Osc_i(\cL f) \leqslant \sum_{j=1}^d A_{ij}   \Osc_j( f) \qquad \forall i = 1,\cdots,d, 
\end{equation}
where $A\in \mathbb{R}^{d\times d}$ is given by:
\begin{equation}\label{A_band}
    A_{ij} = \begin{cases}
      \tilde{r} :=   2|\cX| r_{\max} + 2\bar{r} & \text{~for~} |i-j| \leqslant \kappa \\ 
       0 & \text{~for~} |i-j|> \kappa 
    \end{cases}.
\end{equation}

\noindent
\textbf{Step 2}:
Upper bound of $\Osc_i(\cP_t f)$.

First, we want to compute the ``derivative" of $\Osc_i(\cP_t f)$ w.r.t. $t$. Since $\Osc_i(\cP_t f)$ in general is not differentiable w.r.t $t$, we thus consider its Dini derivative:
\begin{equation*}
    \begin{split}
       D^+ \Osc_i(\cP_t f) :=  \limsup_{\epsilon \rightarrow 0^+} \frac{\Osc_i(\cP_{t+\epsilon} f) - \Osc_i(\cP_t f)}{\epsilon}.
    \end{split}
\end{equation*}
By \Cref{lemma:Dini},
\begin{displaymath}
    D^+ \Osc_i(\cP_t f) \leqslant   \Osc_i(\cL  \cP_t f) .
\end{displaymath}
Use \eqref{LR_bd_step_1} proved in \textbf{Step 1}, we have 
\begin{equation}\label{eq:dini}
    D^+ g_i(t) \leqslant \sum_{j=1}^d  A_{ij} g_j(t).
\end{equation}
where we denote $g_i(t) := \Osc_i(\cP_t f)$ for simplicity.

To proceed, let $g(t)\in \R^{d}$ be the vector field formed by all $g_i(t)$. In the rest of \textbf{Step 2}, we show the following vector Gr\"onwall's inequality  holds componentwisely: 
\begin{equation}\label{eq:vec gronwall}
    g(t) \leqslant e^{tA} g(0) \text{~for any~} t \geqslant  0.
\end{equation}
We first define  
\begin{align*}
    y(t) = \ &  e^{tA} g(0), \\
    h(t) = \ &  g(t) - y(t),\\
    w(t) = \ & \max_i h_i(t).
\end{align*}
Since $y(t)$ is differentiable, $D^+y_i(t)=y_i'(t)=(Ay(t))_i$. By \eqref{eq:dini}, 
\begin{equation*}
    \begin{split}
      D^+ h_i(t) = \ &  D^+ g_i(t) - D^+ y_i(t) \leqslant  (Ag(t))_i - (Ay(t))_i \\ = \ &  \sum_{j=1}^d A_{ij}(g_j(t) - y_j(t)) 
     =   \sum_{j=1}^d A_{ij}h_j(t) \leqslant  \sum_{j=1}^d A_{ij} w(t)
    \end{split}
\end{equation*}
where we have used $A_{ij}\geqslant 0$ for the last inequality above. Therefore, 
\begin{align*}
    D^+w(t) = \ &  \limsup_{\epsilon\rightarrow 0^+} \frac{ \max_i h_i(t+\epsilon) - \max_{i'} h_{i'} (t)}{\epsilon} \\
    \leqslant \ &  \limsup_{\epsilon\rightarrow 0^+} \max_i \frac{  h_i(t+\epsilon) -  h_{i} (t)}{\epsilon} \\ 
    = \ &  \max_i \limsup_{\epsilon\rightarrow 0^+} \frac{  h_i(t+\epsilon) -  h_{i} (t)}{\epsilon} =  \max_i D^+h_i(t) \leqslant \left( \max_i \sum_{j=1}^dA_{ij} \right) w(t).
\end{align*}
Since $w(0) = 0$, by the standard Gr\"onwall's inequality, we have $w(t) \leqslant  0$ for all $t\geqslant 0$. This yields \eqref{eq:vec gronwall} by 
\begin{displaymath}
     g_i(t) \leqslant y_i(t)  =  \sum_{j=1}^d (e^{tA})_{ij} g_j(0). 
\end{displaymath}
Since $g_i(0) = \Osc_i(f)$, we have the following bound for $\Osc_i(\cP_t f)$:
\begin{equation}\label{LR_bd_step_2}
  \Osc_i(\cP_t f) =   g_i(t) \leqslant y_i(t) = \sum_{j=1}^d (e^{tA})_{ij} g_j(0) =   \sum_{j=1}^d (e^{tA})_{ij}  \Osc_i( f). 
\end{equation}

\noindent
\textbf{Step 3}: Upper bound of $(e^{tA})_{ij}$.

Since $A_{ij} =0$ for $|i-j|>\kappa$, 
it is easy to check that $A^n_{ij} =0$ for $|i-j|>n \kappa$. Hence, 
\begin{displaymath}
    (e^{tA})_{ij} = \sum_{n\geqslant \bar n}^\infty \frac{t^n}{n!}(A^n)_{ij} \leqslant \sum_{n\geqslant \bar n}^\infty \frac{t^n}{n!}\|A^n\|_\infty \leqslant  \sum_{n\geqslant \bar n}^\infty\frac{t^n} {n!}\|A\|^n_\infty = \mathbb{P}[Y \geqslant \bar n]
\end{displaymath}
where $\bar n = \lceil |i-j| / \kappa \rceil$ and Poisson random variable $Y\sim \text{Pois}(\bar a t)$ with $\bar a = \|A\|_\infty = \max_i \sum_{j}A_{ij}$. By Markov's inequality and generating function for Poisson random variable, for any $\theta >0$, 
\begin{displaymath}
     \mathbb{P}[Y \geqslant \bar n] =  \mathbb{P}[e^{\theta Y} \geqslant e^{\theta \bar n}] \leqslant \frac{\mathbb{E}[e^{\theta Y} ]}{e^{\theta \bar n}} = \frac{e^{\bar a t(e^\theta -1)}}{e^{\theta \bar n }}=e^{\bar a t(e^\theta - 1) - \theta \bar n}.
\end{displaymath}
Since $\bar n \geqslant  |i-j| / \kappa$ and $\bar a \leqslant (2\kappa+1)\tilde r$ respectively by their definitions, we have  
\begin{equation}\label{LR_bd_step_3}
     (e^{tA})_{ij} \leqslant e^{- \theta \bar n + \bar a t(e^\theta - 1) } \leqslant e^{-\theta\frac{|i-j|}{\kappa} + (2\kappa+1)\tilde rt (e^\theta - 1)} := e^{-\frac{|i-j| - vt}{\kappa/\theta}}
\end{equation}
where
$$
    v = \frac{(2\kappa+1)\kappa\tilde r (e^\theta - 1)}{\theta}
$$
for any $\theta>0$. For convenience, we simply choose $\theta = 1$. Combining \eqref{LR_bd_step_2} and \eqref{LR_bd_step_3}, we obtain the desired estimation stated in \Cref{prop:LR bound}.
\end{proof}

With the Lieb-Robinson bound stated in \Cref{prop:LR bound}, we can prove the decay of correlation for the Dirichlet form as follows: 
\begin{lemma}
\label{lemma:propagation_overlap_bd}
Given two-cluster basis functions $\psi_{i}$ and $\psi_j$ for $i,j\in[(dn)^2]$, under \Cref{assum:finite range} and \Cref{assum:spec gap}, for any lag time $t>0$, we have 
    \begin{equation}
       | \cE(\cP_t \psi_i, \cP_t \psi_j) |\leqslant  \frac{1}{2}|\cX|r_{\max} \min\left(  e^{-2\lambda t}  ,  d e^{-\frac{\xi(i,j)  -2vt}{\kappa}} \right) \|\psi_{i}\|_{\osc} \|\psi_{j}\|_{\osc}. 
    \end{equation}    
\end{lemma}
\begin{proof}
    We prove two upper bounds for $|\cE(\cP_t \psi_{i}, \cP_t \psi_{j})|.$ For the first one, we notice that applying the Cauchy-Schwarz inequality~\eqref{eqn:cauchy schwarz} and Poincar\'e  inequality~\eqref{eqn:poincare} give us 
    \begin{align*}
        |\cE(\cP_t \psi_{i}, \cP_t \psi_{j})|    \leqslant 
        \sqrt{\cE(\cP_t \psi_{i}, \cP_t \psi_{i}) \cE(\cP_t \psi_{j}, \cP_t \psi_{j}) } 
        &\leqslant e^{-2\lambda t} \sqrt{\cE( \psi_{i},  \psi_{i}) \cE( \psi_{j},  \psi_{j}) }\\
        &=e^{-2\lambda t} \sqrt{ \sum_{k=1}^d \cE_k( \psi_{i},  \psi_{i}) \sum_{k'=1}^d\cE_{k'}( \psi_{j},  \psi_{j}) }.
    \end{align*}
    Then, we can use the local energy bound~\eqref{eqn:local_energy_bound} to show 
    \begin{equation}
    \label{prop_upper_bd_1}
    \begin{split}
        |\cE(\cP_t \psi_{i}, \cP_t \psi_{j})|   & \leqslant  \frac{1}{2}|\cX|r_{\max} e^{-2\lambda t} \sqrt{\sum_{k=1}^d \Osc_k ( \psi_{i})^2}\sqrt{\sum_{k'=1}^d \Osc_{k'} ( \psi_{j})^2} \\ 
        &\leqslant    \frac{1}{2}|\cX|r_{\max} e^{-2\lambda t} \|\psi_{i}\|_{\osc} \|\psi_{j}\|_{\osc}. 
    \end{split}
    \end{equation}

   For the second one, we first decompose the Dirichlet form and  again use the Cauchy-Schwarz inequality~\eqref{eqn:cauchy schwarz} and local energy bound~\eqref{eqn:local_energy_bound}  to obtain
   \begin{align*}
       |\cE(\cP_t \psi_{i}, \cP_t \psi_{j})| \leqslant \sum_{k=1}^d |\cE_k(\cP_t \psi_{i}, \cP_t \psi_{j}) |  
       &\leqslant  \sum_{k=1}^d \sqrt{\cE_k(\cP_t \psi_{i}, \cP_t \psi_{i}) \cE_k(\cP_t \psi_{j}, \cP_t \psi_{j}) } \\
        &\leqslant  \frac{1}{2}|\cX|r_{\max} \sum_{k=1}^d   \Osc_k (\cP_t \psi_{i})   \Osc_k (\cP_t \psi_{j}) .
   \end{align*}
Then we use the Lieb-Robinson bound~\eqref{eqn:LR_bound} to get 
   \begin{align*}
       |\cE(\cP_t \psi_{i}, \cP_t \psi_{j})| &\leqslant   \frac{1}{2}|\cX|r_{\max}   \sum_{l=1}^d  \sum_{k,k'=1}^d e^{-\frac{|k-l|+|k'-l|-2vt}{\kappa}}  \Osc_{k}(\psi_{i}) \Osc_{k'}(\psi_{j}) \\
       &=  \frac{1}{2}|\cX|r_{\max}   \sum_{l=1}^d  \sum_{k\in \{s_1(i),s_2(i)\}}  \sum_{k'\in \{s_1(j),s_2(j)\}}  e^{-\frac{|k-l|+|k'-l|-2vt}{\kappa}}  \Osc_{k}(\psi_{i}) \Osc_{k'}(\psi_{j}).
   \end{align*}
The last equality is due to the fact that $\psi_{i}$ and $\psi_{j}$ only act on dimensions $s(i)$ and $s(j)$ respectively. By triangle inequality, we have $|k-l|+|k'-l| \geqslant |k-k'|$. Hence,
\begin{equation}\label{prop_upper_bd_2}
 \begin{split}
    |\cE(\cP_t \psi_{i}, \cP_t \psi_{j})| \leqslant  \ &  \frac{1}{2}|\cX|r_{\max} \sum_{l=1}^d  e^{-\frac{|k-k'| -2vt}{\kappa}} \sum_{k\in \{s_1(i),s_2(i)\}}  \sum_{k'\in \{s_1(j),s_2(j)\}} \Osc_{k}(\psi_{i}) \Osc_{k'}(\psi_{j}) \\ 
    = \ &  \frac{1}{2}|\cX|r_{\max} d \cdot  \max_{k\in \{s_1(i),s_2(i)\}}  \max_{k'\in \{s_1(j),s_2(j)\}}e^{-\frac{|k-k'| -2vt}{\kappa}} \cdot    \|\psi_{i}\|_{\osc} \|\psi_{j}\|_{\osc} .
    \end{split}
\end{equation}
The desired result is immediately obtained by combining the two upper bounds~\eqref{prop_upper_bd_1} and~\eqref{prop_upper_bd_2}.
\end{proof}

At this point, we are able to prove the decay-of-correlation property in~\Cref{lemma:decay of correlation}, which is a direct consequence of \Cref{lemma:propagation_overlap_bd}. 
\begin{proof}
By the self-adjointness property \eqref{self-adjointness},
\begin{align*}
\frac{\d}{\d u}\langle \psi_{i} , \cP_u \psi_j \rangle_{\rho_\infty} =  \langle \psi_i ,  \cP_u \cL \psi_j \rangle_{\rho_\infty}=\langle \cP_{\frac{u}{2}} \psi_i ,  \cP_{\frac{u}{2}} \cL \psi_j \rangle_{\rho_\infty} 
=     \langle \cP_{\frac{u}{2}} \psi_i ,   \cL \cP_{\frac{u}{2}} \psi_j \rangle_{\rho_\infty} = -\cE(\cP_{\frac{u}{2}} \psi_i,\cP_{\frac{u}{2}} \psi_j).
\end{align*}
Taking integration $\int^\infty_t \d u$ on both sides yields:
\begin{align*}
\langle \psi_i, \cP_{t} \psi_j \rangle_{\rho_\infty} - \langle \psi_i \rangle_{\rho_\infty}  \langle  \psi_j \rangle_{\rho_\infty}   =   \int^\infty_{t} \cE(\cP_{\frac{u}{2}} \psi_i,\cP_{\frac{u}{2}} \psi_j)\d u 
= 2 \int^\infty_{\frac{t}{2}} \cE(\cP_{u} \psi_i,\cP_{u} \psi_j) \d u
\end{align*}
where the first term is due to $\langle \psi_i, \cP_{\infty} \psi_j \rangle_{\rho_\infty} = \langle \psi_i, \langle \psi_j  \rangle_{\rho_\infty} \rangle_{\rho_\infty}=\langle \psi_i \rangle_{\rho_\infty}\langle \psi_j  \rangle_{\rho_\infty}$. 

By \Cref{lemma:propagation_overlap_bd}, we have 
\begin{equation}\label{decay of corr 1}
\begin{split}
   | \langle \psi_i, \cP_{t} \psi_j \rangle_{\rho_\infty} - \langle \psi_i \rangle_{\rho_\infty}  \langle  \psi_j \rangle_{\rho_\infty} |  \leqslant  d|\cX|r_{\max}  \int_{\frac{t}{2}}^\infty  \min\left(  e^{-2\lambda u}  ,   e^{-\frac{\xi(i,j)  -2vu}{\kappa}} \right) \d u  
    \cdot  \|\psi_i\|_{\osc} \|\psi_j\|_{\osc}.
   \end{split}
\end{equation}
Note that $f_1(u) := e^{-2\lambda u}$ is a decreasing function, $f_2(u) :=  e^{-\frac{\xi(i,j)  -2vu}{\kappa}}$ is an increasing function, and $1 = f_1(0)\geqslant f_2(0) =  e^{-\frac{\xi(i,j) }{\kappa}}$. Let $t_\star$ be the time when the graphs of $f_1$ and $f_2$ intersect, i.e., $f_1(t_\star) = f_2(t_\star)$, which by direct calculation is 
\begin{displaymath}
    t_\star = \frac{\xi(i,j)}{2(v+\lambda \kappa)}. 
\end{displaymath}
If $t\leqslant 2t_\star$, we can split the integral in \eqref{decay of corr 1} into two parts:
\begin{align*}
    \int_{\frac{t}{2}}^\infty \min\left(  e^{-2\lambda u}  ,  e^{-\frac{\xi(i,j)  -2v u}{\kappa}} \right) \d u = \ &  \int_{\frac{t}{2}}^{t_\star}    e^{-\frac{\xi(i,j)  -2vu}{\kappa}} \d u +  \int_{t_\star}^\infty e^{-2\lambda u}  \d u \\
    = \ & \frac{\kappa}{2v}\left( e^{-\frac{\xi(i,j)-2v t_\star}{\kappa}} - e^{-\frac{\xi(i,j)-v t}{\kappa}} \right) + \frac{1}{2\lambda} e^{-2\lambda t_\star} \\
    \leqslant \ & \frac{\kappa}{2v}f_1(t_\star) + \frac{1}{2\lambda} f_2(t_\star) \leqslant \max\left(\frac{1}{\lambda},\frac{\kappa}{v}\right) e^{-2\lambda t_\star};
\end{align*}
if $t >  2t_\star$, 
\begin{displaymath}
     \int_{\frac{t}{2}}^\infty \min\left(  e^{-2\lambda u}  ,  e^{-\frac{\xi(i,j)  -2v u}{\kappa}} \right) \d u =    \int_{\frac{t}{2}}^\infty e^{-2\lambda u}  \d u = \frac{1}{2\lambda}e^{-t\lambda} \leqslant \frac{1}{2\lambda} e^{-2\lambda t_\star}.
\end{displaymath}
Insert the formula for $t_\star$, we arrive at the uniform upper bound
\begin{align*}
     \int_{\frac{t}{2}}^\infty \min\left(  e^{-2\lambda u}  ,  e^{-\frac{\xi(i,j)  -2v u}{\kappa}} \right) \d u & \leqslant \max\left(\frac{1}{\lambda},\frac{\kappa}{v}\right) e^{-\frac{\lambda \xi(i,j) }{v+\lambda \kappa}},
\end{align*}
 leading to the final estimation in \Cref{lemma:decay of correlation}.
\end{proof}

\printbibliography

@book{bakry2014analysis,
  title={Analysis and Geometry of Markov Diffusion Operators},
  author={Bakry, Dominique and Gentil, Ivan and Ledoux, Michel},
  year={2014},
  publisher={Springer},
  isbn={9783319002262}
}

@article{E2005242,
title = {Transition pathways in complex systems: Reaction coordinates, isocommittor surfaces, and transition tubes},
journal = {Chem. Phys. Lett.},
volume = {413},
number = {1},
pages = {242-247},
year = {2005},
author = {Weinan E and Weiqing Ren and Eric Vanden-Eijnden},
}

@article{pmlr-v145-li22a,
  title = 	 {A semigroup method for high dimensional committor functions based on neural network},
  author =       {Li, Haoya and Khoo, Yuehaw and Ren, Yinuo and Ying, Lexing},
  pages = 	 {598--618},
  year = 	 {2022},
  volume = 	 {145},
  month = 	 {8},
  journal =    {PMLR}
}

@article{
doi:10.1073/pnas.0803205106,
author = {Michael W. Mahoney  and Petros Drineas },
title = {{CUR} matrix decompositions for improved data analysis},
journal = {Proc. Natl. Acad. Sci. U.S.A.},
volume = {106},
number = {3},
pages = {697-702},
year = {2009},
}

@article{10.1063/1.2013256,
    author = {Ren, Weiqing and Vanden-Eijnden, Eric and Maragakis, Paul and E, Weinan},
    title = {Transition pathways in complex systems: Application of the finite-temperature string method to the alanine dipeptide},
    journal = {J. Chem. Phys.},
    volume = {123},
    number = {13},
    pages = {134109},
    year = {2005},
}

@article{E2010TransitionpathTA,
  title={Transition-path theory and path-finding algorithms for the study of rare events.},
  author={Weinan E and Eric Vanden-Eijnden},
  journal={Annu. Rev. Phys. Chem.},
  year={2010},
  volume={61},
  pages={
          391-420
        }
}

@article{E2006,
  author    = {Weinan E and Eric Vanden-Eijnden},
  title     = {Towards a Theory of Transition Paths},
  journal   = {J. Stat. Phys.},
  volume    = {123},
  number    = {3},
  pages     = {503--523},
  year      = {2006},
}

@book{hoffmann2012GL,
  title={Ginzburg-Landau Phase Transition Theory and Superconductivity},
  author={Hoffmann, Karl-Heinz and Tang, Qi},
  year={2012},
  publisher={Birkhäuser Basel}
}

@article{Goreinov2001TheMC,
  title={The maximum-volume concept in approximation by low-rank matrices},
  author={Sergei A. Goreinov and Eugene E. Tyrtyshnikov},
  year={2001},
  journal={Contemp. Math.},
  volume = {280},
pages = {41-57},
}

@article{submatrix,
author = {Goreinov, Sergei and Oseledets, Ivan and Savostyanov, D. and Tyrtyshnikov, E. and Zamarashkin, Nickolai},
year = {2010},
month = {4},
title = {How to Find a Good Submatrix},
journal = {Matrix Methods: Theory, Algorithms and Applications},
}

@article{AllenLaiShen2024Maxvol,
  author    = {K. Allen and M. J. Lai and Z. Shen},
  title     = {Maximal volume matrix cross approximation for image compression and least squares solution},
  journal   = {Adv. Comput. Math.},
  volume    = {50},
  pages     = {102},
  year      = {2024},
}

@article{doi:10.1137/24M1654075,
author = {Tang, Xun and Collis, Leah and Ying, Lexing},
title = {Solving High-Dimensional Kolmogorov Backward Equations with Functional Hierarchical Tensor Operators},
journal = {SIAM J. Sci. Comput.},
volume = {47},
number = {2},
pages = {A613-A632},
year = {2025},
}

@article{peng2025,
      title={Tensor Density Estimator by Convolution-Deconvolution}, 
      author={Yifan Peng and Siyao Yang and Yuehaw Khoo and Daren Wang},
      year={2025},
      journal={arXiv preprint arXiv:2412.18964},
}

@article{chen2023,
      title={Combining Monte Carlo and Tensor-network Methods for Partial Differential Equations via Sketching}, 
      author={Yian Chen and Yuehaw Khoo},
      year={2023},
      journal={arXiv preprint arXiv:2305.17884},
}

@article{doi:10.1137/090771806,
author = {Halko, N. and Martinsson, P. G. and Tropp, J. A.},
title = {Finding Structure with Randomness: Probabilistic Algorithms for Constructing Approximate Matrix Decompositions},
journal = {SIAM Review},
volume = {53},
number = {2},
pages = {217-288},
year = {2011},
}

@article{strahan2021long,
  title={Long-time-scale predictions from short-trajectory data: A benchmark analysis of the trp-cage miniprotein},
  author={Strahan, John and Antoszewski, Adam and Lorpaiboon, Chatipat and Vani, Bodhi P and Weare, Jonathan and Dinner, Aaron R},
  journal={J. Chem. Theory Comput},
  volume={17},
  number={5},
  pages={2948--2963},
  year={2021},
}

@article{thiede2019galerkin,
  title={Galerkin approximation of dynamical quantities using trajectory data},
  author={Thiede, Erik H and Giannakis, Dimitrios and Dinner, Aaron R and Weare, Jonathan},
  journal={J. Chem. Phys.},
  volume={150},
  number={24},
  year={2019},
}

@article{peng2023,
      title={Generative Modeling via Hierarchical Tensor Sketching}, 
      author={Yifan Peng and Yian Chen and E. Miles Stoudenmire and Yuehaw Khoo},
      year={2023},
      journal={arXiv preprint arxiv:2304.05305},
}

@article{williams2000using,
  title={Using the Nystr{\"o}m method to speed up kernel machines},
  author={Williams, Christopher and Seeger, Matthias},
  journal={NeurIPS},
  volume={13},
  year={2000}
}

@article{doi:10.1137/21M1401243,
author = {Brunton, Steven L. and Budi\v{s}i\'{c}, Marko and Kaiser, Eurika and Kutz, J. Nathan},
title = {Modern Koopman Theory for Dynamical Systems},
journal = {SIAM Review},
volume = {64},
number = {2},
pages = {229-340},
year = {2022},
}

@article{khoo2025optimizationfreediffusionmodel,
      title={Optimization-Free Diffusion Model -- A Perturbation Theory Approach}, 
      author={Yuehaw Khoo and Mathias Oster and Yifan Peng},
      year={2025},
      journal={arXiv preprint arXiv:2505.23652}, 
}

@article{doi:10.1098/rspa.2024.0001,
author = {Chen, Yian  and Khoo, Yuehaw  and Lim, Lek-Heng },
title = {Convex relaxation for Fokker–Planck equation},
journal = {Proc. R. Soc. A.},
volume = {481},
number = {2313},
pages = {20240001},
year = {2025},
}

@article{drautz2019atomic,
  title={Atomic cluster expansion for accurate and transferable interatomic potentials},
  author={Drautz, Ralf},
  journal={Phys. Rev. B},
  volume={99},
  number={1},
  pages={014104},
  year={2019},
  publisher={APS}
}

@article{dusson2022atomic,
  title={Atomic cluster expansion: Completeness, efficiency and stability},
  author={Dusson, Genevieve and Bachmayr, Markus and Cs{\'a}nyi, G{\'a}bor and Drautz, Ralf and Etter, Simon and van Der Oord, Cas and Ortner, Christoph},
  journal={J. Comput. Phys.},
  volume={454},
  pages={110946},
  year={2022},
  publisher={Elsevier}
}

@article{meneghelli2014mayer,
  title={Mayer-cluster expansion of instanton partition functions and thermodynamic Bethe ansatz},
  author={Meneghelli, Carlo and Yang, Gang},
  journal={J. High Energy Phys},
  volume={2014},
  number={5},
  pages={1--42},
  year={2014},
  publisher={Springer}
}

@article{brydges1984short,
  title={A short course on cluster expansions},
  author={Brydges, David C},
  journal={Les Houches},
  number={PART I},
  year={1984}
}

@article{10.1214/aop/1176993067,
author = {Richard Holley},
title = {{Rapid Convergence to Equilibrium in One Dimensional Stochastic Ising Models}},
volume = {13},
journal = {Ann. Probab.},
number = {1},
pages = {72 -- 89},
year = {1985},
}

@article{Caputo2004SpectralGap,
  author    = {Pietro Caputo},
  title     = {Spectral Gap Inequalities in Product Spaces with Conservation Laws},
  journal={Adv. Stud. Pure Math.},
  volume    = {39},
  pages     = {53--88},
  year      = {2004},
}

@article{10.1214/aop/1041903209,
author = {C. Landim and S. Sethuraman and S. Varadhan},
title = {{Spectral gap for zero-range dynamics}},
volume = {24},
journal = {Ann. Probab.},
number = {4},
pages = {1871 -- 1902},
year = {1996},
}

@article{cmp/1104253633,
author = {Sheng Lin Lu and Horng-Tzer Yau},
title = {{Spectral gap and logarithmic Sobolev inequality for Kawasaki and Glauber dynamics}},
volume = {156},
journal = {Comm. Math. Phys.},
number = {2},
pages = {399 -- 433},
year = {1993},
}

@article{liggett1997interacting,
  title={Interacting Particle Systems},
  author={Liggett, Thomas M},
  journal={Directorate for Mathematical and Physical Sciences},
  volume={97},
  number={9703830},
  pages={3830},
  year={1997}
}

@article{capel2025decay,
  title={From decay of correlations to locality and stability of the Gibbs state},
  author={Capel, {\'A}ngela and Moscolari, Massimo and Teufel, Stefan and Wessel, Tom},
  journal={Commun. Math. Phys.},
  volume={406},
  number={2},
  pages={43},
  year={2025},
  publisher={Springer}
}

@article{chen2023speed,
  title={Speed limits and locality in many-body quantum dynamics},
  author={Chen, Chi-Fang Anthony and Lucas, Andrew and Yin, Chao},
  journal={Rep.  Prog. Phys.},
  volume={86},
  number={11},
  pages={116001},
  year={2023},
  publisher={IOP Publishing}
}

@article{GOREINOV19971,
title = {A theory of pseudoskeleton approximations},
journal = {Linear Algebra Appl.},
volume = {261},
number = {1},
pages = {1-21},
year = {1997},
author = {S.A. Goreinov and E.E. Tyrtyshnikov and N.L. Zamarashkin},
}

@article{CORTINOVIS2020251,
title = {On maximum volume submatrices and cross approximation for symmetric semidefinite and diagonally dominant matrices},
journal = {Linear Algebra Appl.},
volume = {593},
pages = {251-268},
year = {2020},
author = {Alice Cortinovis and Daniel Kressner and Stefano Massei},
}

@article{doi:10.1137/0328056,
author = {Kushner, Harold J.},
title = {Numerical Methods for Stochastic Control Problems in Continuous Time},
journal = {SIAM J. Control Optim.},
volume = {28},
number = {5},
pages = {999-1048},
year = {1990},
}

@incollection{KOVACHKI2024419,
title = {Operator learning: Algorithms and analysis},
series = {Handbook of Numerical Analysis},
publisher = {Elsevier},
volume = {25},
pages = {419-467},
year = {2024},
booktitle = {Numerical Analysis Meets Machine Learning},
author = {Nikola B. Kovachki and Samuel Lanthaler and Andrew M. Stuart},
}

@incollection{BOULLE202483,
title = {A mathematical guide to operator learning},
editor = {Siddhartha Mishra and Alex Townsend},
series = {Handbook of Numerical Analysis},
publisher = {Elsevier},
volume = {25},
pages = {83-125},
year = {2024},
booktitle = {Numerical Analysis Meets Machine Learning},
author = {Nicolas Boullé and Alex Townsend},
}

@article{Schtte1999ADA,
  title={A Direct Approach to Conformational Dynamics Based on Hybrid {Monte Carlo}},
  author={Christof Sch{\"u}tte and Alexander Fischer and Wilhelm Huisinga and Peter Deuflhard},
  journal={J. Comput. Phys.},
  year={1999},
  volume={151},
  pages={146-168},
}

@article{Pirvu_2010,
year = {2010},
month = {2},
volume = {12},
number = {2},
pages = {025012},
author = {Pirvu, B and Murg, V and Cirac, J I and Verstraete, F},
title = {Matrix product operator representations},
journal = {New J. Phys.},
}

@article{PhysRevA.74.022320,
  title = {Classical simulation of quantum many-body systems with a tree tensor network},
  author = {Shi, Y.-Y. and Duan, L.-M. and Vidal, G.},
  journal = {Phys. Rev. A},
  volume = {74},
  issue = {2},
  pages = {022320},
  numpages = {4},
  year = {2006},
  month = {8},
}

@article{Lucke2022tgEDMD,
  author    = {L{\"u}cke, Moritz and N{\"u}ske, Feliks},
  title     = {{tgEDMD}: Approximation of the Kolmogorov Operator in Tensor Train Format},
  journal   = {J. Nonlinear Sci.},
  volume    = {32},
  number    = {44},
  year      = {2022},
}

@article{PhysRevA.81.062337,
  title = {Tensor operators: Constructions and applications for long-range interaction systems},
  author = {Fr\"owis, F. and Nebendahl, V. and D\"ur, W.},
  journal = {Phys. Rev. A},
  volume = {81},
  issue = {6},
  pages = {062337},
  numpages = {18},
  year = {2010},
  month = {6},
}

@article{khoromskij2011quantics,
  author    = {Khoromskij, Boris N.},
  title     = {$O(d \log N)$-Quantics Approximation of $N-d$ Tensors in High-Dimensional Numerical Modeling},
  journal   = {Const. Approx.},
  volume    = {34},
  number    = {2},
  pages     = {257--280},
  year      = {2011},
  publisher = {Springer},
}

@article{doi:10.1021/jp037421y,
author = {Swope, William C. and Pitera, Jed W. and Suits, Frank},
title = {Describing Protein Folding Kinetics by Molecular Dynamics Simulations. 1. Theory},
journal = {J. Phys. Chem. B},
volume = {108},
number = {21},
pages = {6571-6581},
year = {2004},
}

@article{PANDE201099,
title = {Everything you wanted to know about Markov State Models but were afraid to ask},
journal = {Methods},
volume = {52},
number = {1},
pages = {99-105},
year = {2010},
author = {Vijay S. Pande and Kyle Beauchamp and Gregory R. Bowman}
}

@article{osti_2281727,
  author       = {Lu, Lu and Jin, Pengzhan and Pang, Guofei and Zhang, Zhongqiang and Karniadakis, George Em},
  title        = {Learning nonlinear operators via {DeepONet} based on the universal approximation theorem of operators},
  journal      = {Nat. Mach. Intell.},
  number       = {3},
  volume       = {3},
  year         = {2021},
  month        = {03}}

@misc{Kressner,
author={Kressner, Daniel},
title={Low-Rank Approximation, Lecture 4},
url= {https://www.epfl.ch/labs/anchp/wp-content/uploads/2018/10/lecture4-slides.pdf}
}

@article{CHEN2023111646,
title = {Committor functions via tensor networks},
journal = {J. Comput. Phys.},
volume = {472},
pages = {111646},
year = {2023},
author = {Yian Chen and Jeremy Hoskins and Yuehaw Khoo and Michael Lindsey},
}

@article{song2020score,
  title={Score-based generative modeling through stochastic differential equations},
  author={Song, Yang and Sohl-Dickstein, Jascha and Kingma, Diederik P and Kumar, Abhishek and Ermon, Stefano and Poole, Ben},
  journal={arXiv preprint arXiv:2011.13456},
  year={2020}
}

@article{pmlr-v235-chen24n,  
title =  {Probabilistic Forecasting with Stochastic Interpolants and Föllmer Processes},  
author =       {Chen, Yifan and Goldstein, Mark and Hua, Mengjian and Albergo, Michael Samuel and Boffi, Nicholas Matthew and Vanden-Eijnden, Eric},  
pages =  {6728--6756},  
year =  {2024},   
volume =  {235},  
journal =  {Proc. Mach. Learn. Res.},  
month =  {7},}

@article{D0SC03635H,
author ={Sidky, Hythem and Chen, Wei and Ferguson, Andrew L.},
title  ={Molecular latent space simulators},
journal  ={Chem. Sci.},
year  ={2020},
volume  ={11},
pages  ={9459-9467},
}

@article{ho2020denoising,
  title={Denoising diffusion probabilistic models},
  author={Ho, Jonathan and Jain, Ajay and Abbeel, Pieter},
  journal={Advances in neural information processing systems},
  volume={33},
  pages={6840--6851},
  year={2020}
}

@article{coifman2006diffusion,
  title={Diffusion maps},
  author={Coifman, Ronald R and Lafon, St{\'e}phane},
  journal={Appl. Comput. Harmon.},
  volume={21},
  number={1},
  pages={5--30},
  year={2006},
  publisher={Elsevier}
}

@article{li2020fourier,
  title={Fourier neural operator for parametric partial differential equations},
  author={Li, Zongyi and Kovachki, Nikola and Azizzadenesheli, Kamyar and Liu, Burigede and Bhattacharya, Kaushik and Stuart, Andrew and Anandkumar, Anima},
  journal={ICLR},
  year={2021}
}

@article{hesthaven2018non,
  title={Non-intrusive reduced order modeling of nonlinear problems using neural networks},
  author={Hesthaven, Jan S and Ubbiali, Stefano},
  journal={J. Comput. Phys.},
  volume={363},
  pages={55--78},
  year={2018},
}

@article{Kobyzev2020NormalizingFA,
  title={Normalizing Flows: An Introduction and Review of Current Methods},
  author={Ivan Kobyzev and Simon Prince and Marcus A. Brubaker},
  journal={IEEE Trans. Pattern Anal. Mach. Intell},
  year={2020},
  volume={43},
  pages={3964-3979},

}

@article{pmlr-v37-rezende15,  
title =  {Variational Inference with Normalizing Flows},  
author =  {Rezende, Danilo and Mohamed, Shakir},  
pages =  {1530--1538},  
year =  {2015}, 
volume =  {37},  
journal =  {Proc. Mach. Learn. Res.},   month =  {07}   
}

@article{zhang2025probabilisticoperatorlearninggenerative,
      title={Probabilistic operator learning: generative modeling and uncertainty quantification for foundation models of differential equations}, 
      author={Benjamin J. Zhang and Siting Liu and Stanley J. Osher and Markos A. Katsoulakis},
      year={2025},
      journal={arXiv preprint: arXiv 2509.05186},
}

@article{funaki1869stochastic,
  title={Stochastic interface models},
  author={Funaki, Tadahisa},
  journal={Lectures on probability theory and statistics},
  volume={1869},
  pages={103--274},
  year={1869},
  publisher={Springer}
}

@article{PhysRevLett.104.190401,
  title = {Lieb-Robinson Bound and Locality for General Markovian Quantum Dynamics},
  author = {Poulin, David},
  journal = {Phys. Rev. Lett.},
  volume = {104},
  issue = {19},
  pages = {190401},
  numpages = {4},
  year = {2010},
  month = {05},
}

@inproceedings{savostyanov2011fast,
  title={Fast adaptive interpolation of multi-dimensional arrays in tensor train format},
  author={Savostyanov, Dmitry and Oseledets, Ivan},
  booktitle={The 2011 International Workshop on Multidimensional (nD) Systems},
  pages={1--8},
  year={2011},
  organization={IEEE}
}

\end{document}